\documentclass[a4paper,reqno]{amsart}

\usepackage{amssymb}
\usepackage{latexsym}
\usepackage{amsmath}
\usepackage{euscript}
\usepackage{bbm}

\usepackage{tikz}
\usepackage{float}
\usetikzlibrary{automata,positioning,arrows.meta}
\usepackage{pgfplots}
\usepgfplotslibrary{polar}

\usepackage[normalem]{ulem}

\def\cD{{\mathcal D}}   \def\cE{{\mathcal E}}   
\def\cG{{\mathcal G}}      
\def\cJ{{\mathcal J}}      
      
\def\cP{{\mathcal P}}      
      
\def\cV{{\mathcal V}}

\def\cal H{{\mathcal H}}

\def\R{\mathbb{R}}
\def\C{\mathbb{C}}

\def\N{\mathbb{N}}

\def\dom{{\text{\rm dom\,}}}

\def\loc{{\text{\rm loc}}}
\def\dist{{\text{\rm dist}}}

\def\phi{\varphi}
\def\d{\textup{d}}
\def\eps{\varepsilon}

\DeclareMathOperator{\diam}{diam}
\DeclareMathOperator{\osc}{osc}

\DeclareMathOperator{\interior}{int}

\renewcommand{\theta}{\vartheta}

\newtheorem{theorem}{Theorem}[section]
\newtheorem*{thm*}{Theorem}
\newtheorem{proposition}[theorem]{Proposition}
\newtheorem{corollary}[theorem]{Corollary}
\newtheorem{lemma}[theorem]{Lemma}

\newtheorem{conjecture}[theorem]{Conjecture}
\newtheorem{question}[theorem]{Question}

\theoremstyle{definition}
\newtheorem{definition}[theorem]{Definition}
\newtheorem{example}[theorem]{Example}
\newtheorem{assumption}[theorem]{Assumption}

\newtheorem{remark}[theorem]{Remark}
\newtheorem*{ack}{Acknowledgements}

\numberwithin{equation}{section}
\numberwithin{figure}{section}

\title[On the hot spots of quantum graphs]{On the hot spots of quantum graphs}

\author[J.\,B.~Kennedy]{James B.\, Kennedy}
\address{Grupo de F\'isica Matem\'atica\\ Faculdade de Ci\^encias, Universidade de Lisboa \\ Campo Grande, Edif\'icio C6 \\ 1749-016 Lisboa \\ Portugal}
\email{jbkennedy@fc.ul.pt}

\author[J.~Rohleder]{Jonathan Rohleder}
\address{Matematiska institutionen\\ Stockholms universitet \\
106 91 Stockholm \\
Sweden}
\email{jonathan.rohleder@math.su.se}



\begin{document}

\begin{abstract}
We undertake a systematic investigation of the maxima and minima of the eigenfunctions associated with the first nontrivial eigenvalue of the Laplacian on a metric graph equipped with standard (continuity--Kirchhoff) vertex conditions. This is inspired by the famous hot spots conjecture for the Laplacian on a Euclidean domain, and the points on the graph where maxima and minima are achieved represent the generically ``hottest'' and ``coldest'' spots of the graph. We prove results on both the number and location of the hot spots of a metric graph, and also present a large number of examples, many of which run contrary to what one might na\"ively expect. Amongst other results we prove the following: (i) generically, up to arbitrarily small perturbations of the graph, the points where minimum and maximum, respectively, are attained are unique; (ii) the minima and maxima can only be located at the vertices of degree one or inside the doubly connected part of the metric graph; and (iii) for any fixed graph topology, for some choices of edge lengths all minima and maxima will occur only at degree-one vertices, while for others they will only occur in the doubly connected part of the graph.
\end{abstract}

\maketitle

\vspace*{-4mm}

\tableofcontents

\section{Introduction}
\label{sec:intro} 

In recent years there has been a pronounced growth of interest in the structure of the spectrum of quantum graphs, that is, of differential operators such as realisations of the Laplacian defined on metric graphs, see for example \cite{ASSW17,Ari16,BaLe17,BKKM18,BKKM17,BCJ19,Fri05,KKTK16,KKMM16,KMN13,KN14,R17,RS18}. In most of these works particular attention has been given to the relationship between the eigenvalues of such an operator and the topological and metric structure of the graph on which the operator is defined: for example, for the standard Laplacian (i.e., the Laplacian equipped with continuity and Kirchhoff conditions at the vertices), which graph maximises or minimises the first nontrivial eigenvalue among all graphs of fixed total length, or diameter etc.? Even for other variational problems such as nonlinear Schr\"odinger equations on metric graphs, there is now an extensive literature examining the relationship between the topological and metric structure of the graph and the existence of solutions, e.g., \cite{AST17,AST16,AST15,CFN17,DT19,Hof19,KPG19} and the references therein. 

By now such questions of ``shape optimisation'' for eigenvalues (as well as the existence of solutions of nonlinear equations) have been thoroughly investigated. At least as interesting, and as informative, is the behaviour, or profile, of the corresponding \emph{eigenfunctions}, as the following two motivating examples should demonstrate. 

Firstly, one major application of eigenfunctions (say, of the Laplace-Beltrami operator on a manifold), especially those sign-changing ones corresponding to the smallest positive eigenvalue, is that their nodal domains, i.e., the connected components of the set where the eigenfunction is nonzero, tend to be a good way to partition the object on which they are defined: in the case of manifolds, this is the classical observation of Cheeger \cite{Che70}. But more recently a large body of literature has developed around partitions of \emph{discrete} graphs via nodal (zero) and sign-changing properties of the (discrete) Laplacian eigenvectors, in particular the so-called Fiedler vectors, the sign-changing eigenvectors corresponding to the smallest nontrivial eigenvalue, also known as the algebraic connectivity; although many results have now been extended to the higher eigenvalues and so-called higher-order Cheeger constants. We refer to \cite{LLPP15,LGT14} and the references therein. There is a large body of work on the nodal structure of the eigenfunctions of Laplacian-type operators on metric graphs, see, for example, \cite{ABB18,Ban14,BBS12,GSW04} and the references therein, although the focus is perhaps more commonly placed on the size of the nodal set of the eigenfunctions (the so-called nodal count) than the distribution of the zeros. Work is currently underway to investigate spectral partition questions on metric graphs, see \cite{BBRS12,HKMP20,KKLM20}. 

Secondly, the \emph{hot spots conjecture}, originally formulated in the 1970s for domains, roughly speaking asks for which domains (or manifolds, graphs, \ldots) the maximum and minimum of the eigenfunction(s) of the first nontrivial eigenvalue of the Neumann Laplacian are on the boundary of the domain; see, e.g., the introduction of \cite{BB99} for a motivation and an historical description of the problem on domains, \cite{BW99} for the famous counterexample to the original conjecture, and~\cite{JM20,KT19,S15,S20} for recent advances on the problem. The idea behind the conjecture comes from the corresponding heat equation: an expansion of solutions as Fourier series in the eigenfunctions shows that the maximum and minimum of the first nontrivial eigenfunctions represent the generically hottest and coldest points in the domain; and as heat flow should respect the geometry of the domain, it is natural to expect these points to be located far away from each other in some reasonable sense.

The same question, or rather an adapted variant, has also been asked in the case of discrete graphs: motivated in part by previous applications of the critical points of the Fiedler vector(s) of a finite discrete graph to the analysis of data in various contexts, in \cite{CSAV11} the authors conjectured that its points of maximum and minimum should always realise the diameter of the graph; we might call this a version of the hot spots conjecture for discrete graphs. As it turns out, there are fairly simple counterexamples, even among tree graphs, as first exhibited in \cite{E11}. But the principle that the graph hot spots should represent an analytic, weighted version of the diameter is intuitively reasonable, and finding classes of graphs for which the graph version of the hot spots conjecture holds (and understanding better the extent to which it fails) is a topic of ongoing research; we refer to \cite{GP19,Lef13} and in particular \cite{LS19}.

These same motivations remain valid in the case of quantum graphs, where comparatively little seems to be known, at least in terms of the profile of the eigenfunctions: some work has been done constructing so-called \emph{landscape functions} to control their size \cite{HM18a,HM18b}, and relatively recently the concept of \emph{Neumann domains} of the eigenfunctions, the regions separated by critical points of the eigenfunctions, was introduced and is now being studied \cite{Alo20,AB19,ABBE18,BF16}. But to date the ``hot and cold spots'' of a quantum graph do not seem to have received direct attention, a preliminary note of the current authors excluded \cite{KRpamm}. Our principal goal here is thus to understand better how these hot and cold spots, more precisely the global, and also local, maxima and minima of the first nontrivial eigenfunction(s) of the standard Laplacian, the natural quantum graph analogue of both the Neumann and the discrete graph Laplacians, depend on metric, geometric and topological features of the metric graph. The current work is thus an attempt to initiate investigation into precisely this relationship, not just in terms of the location of these extrema but also in terms of their number. Since we believe this to be the first systematic investigation into the subject, more open problems and conjectures arise than we can reasonably deal with in one work. Therefore we will summarise a large number of these in a final section. 

Let us give a short overview on the results and observations of this paper. First of all, we show that the ``na\"ive'' version of the hot spots conjecture, which states that the global maxima and minima of any eigenfunction for the smallest positive eigenvalue of the standard Laplacian are exclusively located at the set of vertices of degree one, proved in \cite{KRpamm} for metric trees (i.e.\ metric graphs without cycles), fails in general. This is obvious for graphs which do not contain vertices of degree one at all, but we also provide examples of graphs with arbitrarily many such vertices where, nevertheless, all the ``hot and cold spots'' lie elsewhere (see Section~\ref{sec:location-summary}). This may be viewed as evidence that the set of degree-one vertices is not necessarily a good notion of a boundary for a metric graph, at least from an analytic point of view -- except for the case of trees. This theme will come to the fore several times over the course of the paper; indeed, intuitively, the low eigenvalues and their eigenfunctions do not ``see'' extremely short edges, thus a large perturbation of the set of degree-one vertices may correspond to a very small perturbation of the eigenvalues and eigenfunctions.

However, our first major result states that the location of (local as well as global) maxima and minima on a metric graph is not totally arbitrary but they have to lie, if not on degree-one vertices, then within the doubly-connected part of the graph; cf.\ Section~\ref{sec:location-summary} and in particular Theorem~\ref{thm:doublyConnectedPart}. In other words, roughly speaking, maxima and minima of such an eigenfunction can only be located either on a vertex of degree one or inside a cycle. It is natural to ask to what extent which of these cases prevails is dependent on the topology of the underlying discrete graph. We observe that for any given discrete graph one may force maximum and minimum to be achieved arbitrarily on vertices of degree one or in the doubly-connected part by choosing appropriate edge lengths for the corresponding metric graph, provided the graph has at least two degree-one vertices and two cycles; see Section~\ref{sec:topology-summary} and in particular Theorem~\ref{thm:examplesummary}.

As regards the relationship between the ``hottest'' and ``coldest'' spots of a metric graph and its diameter, we already noted in \cite{KRpamm} that the distance between these points does not generally realise the diameter, even on trees; indeed, the counterexample in \cite{KRpamm} is very much in the spirit of the ``Fiedler rose'' graph constructed for discrete graphs in~\cite{E11}. In the present paper we sharpen this construction by giving a family of graphs each of which has diameter one and for which the distance between the maxima and the minima of the eigenfunction for the smallest positive eigenvalue becomes arbitrarily small (Proposition~\ref{prop:hotspotsClose}). On the other hand, we show that for some special classes of graphs such as star graphs the distance between any minimum and any maximum equals the diameter (Proposition~\ref{prop:starDiameter}).

Let us now consider the \emph{number} of hot spots of a metric graph. It is clear that there exist only finitely many such points as long as the eigenfunction of the standard Laplacian corresponding to the smallest positive eigenvalue is unique up to scalar multiples. However, this is not always the case, and in cases where it fails, the number of hot spots may be infinite; for instance, for the graph consisting of a single loop every point is the global maximum of some eigenfunction. We show that this phenomenon is not limited to the loop but happens as well for so-called equilateral pumpkins and for equilateral complete graphs (Proposition~\ref{prop:m-gamma-examples}). Using the same tools we show that the set of points where any eigenfunction corresponding to the smallest positive eigenvalue takes a local maximum or minimum is either finite or uncountable (Proposition~\ref{prop:m-finite-or-uncountable}). For the global maximum and minimum we show that after an arbitrarily small perturbation of the graph (in the sense of changing edge lengths arbitrarily little or attaching arbitrarily short pendant edges) these points are unique; more specifically, the eigenspace is one-dimensional and the corresponding eigenfunction takes its global maximum and minimum at exactly one point each (Theorem~\ref{thm:m-generically-2}).

This article is structured as follows. After some preliminaries on metric graphs and their hot spots in Section~\ref{sec:prelim}, we use Section~\ref{sec:examples} to discuss the hot spots of some special classes of graphs such as pumpkins, flowers and complete graphs. This may give the reader a grasp of the possible behaviour of these spots on a metric graph. In Section~\ref{sec:main} we provide a complete overview of all our  results, including the examples, and state them rigorously. They split into results on the number of hot spots (Section~\ref{sec:number-summary}), the location of the hot spots (Section~\ref{sec:location-summary}) and the behaviour of the hot spots as a function of the edge lengths, among all graphs of a given topology (Section~\ref{sec:topology-summary}), including but not limited to the theorems and examples mentioned above. Sections~\ref{sec:number}--\ref{sec:topology} contain the proofs of all the results as well as the elaboration of the examples in detail. Finally, Section~\ref{sec:conjectures} contains a series of conjectures whose proofs would exceed the scope of this paper, together with some additional remarks. As the proofs of some of our results rely on convergence properties of eigenfunctions if certain edge lengths shrink to zero, we provide the required technical result in the appendix.

\section{Preliminaries}%: metric graphs and the standard Laplacian}
\label{sec:prelim}

\subsection{Basics on metric graphs}
\label{subsec:basic}

Throughout this paper, $\cG = (\cV, \cE)$ is a discrete graph consisting of a finite set $\cV$ of vertices and a finite set $\cE$ of edges; we write $E$ and $V$ for the cardinality of $\cE$ and $\cV$, respectively. We normally assume that $\cG$ is connected, i.e., each two vertices are connected to each other by a path. We assign a length $L (e) \in (0,\infty)$ to each edge $e \in \cE$ and identify each edge $e$ with the interval $[0, L (e)]$. Upon taking the natural metric induced by the Euclidean metric on each edge, $\cG$ becomes a compact metric space that we call a {\em metric graph} and denote by $\Gamma$. We will sometimes say that $\cG$ is the {\em underlying discrete graph} corresponding to the metric graph $\Gamma$. Conversely, for a discrete graph $\cG$ we call any metric graph $\Gamma$ whose underlying discrete graph equals $\cG$ an {\em associated metric graph}. Given a discrete graph $\cG$, the set of all possible associated metric graphs may be canonically parametrised by $\R_+^E$: more precisely, we fix an ordering of the set $\cE$ of edges and for each $(y_1,\ldots,y_E) \in \R_+^E$ assign the metric graph with edge lengths $L(e_1) = y_1, \ldots, L(e_E) = y_E$.

The metric on $\Gamma$ gives rise to a distance function $\dist$, with respect to which the diameter of the graph
\begin{displaymath}
	\diam(\Gamma) := \max \{ \dist (x,y): x,y \in \Gamma \}
\end{displaymath}
is well defined and finite. We will denote by $\dist_e$ the corresponding distance function on the edge $e \in \cE$, treated as the interval $[0, L(e)]$. We will, moreover, write
\begin{align*}
 L (\Gamma) := \sum_{e \in \cE} L (e)
\end{align*}
for the total length of $\Gamma$.

As we identify every edge $e \in \cE$ with an interval $[0, L (e)]$, it is natural to distinguish the vertex $o (e)$ from which $e$ originates, i.e.\ which corresponds to the zero endpoint of the interval, and the vertex $t (e)$ at which $e$ terminates. An edge $e$ is called a loop if $o (e) = t (e)$. For a given vertex $v \in \cV$, we set $\cE_{v, \rm i} := \{ e \in \cE: v = o (e)\}$ and $\cE_{v, \rm t} := \{ e \in \cE: v = t (e)\}$ to be the set of all edges that originate from $v$ or terminate at $v$, respectively. The \emph{degree} $\deg v$ of a vertex $v$ is defined as $\deg v = |\cE_{v, \rm i}| + |\cE_{v, \rm t}|$. We explicitly allow $\Gamma$ to have parallel edges and loops.

The {\em (first) Betti number} $\beta$ of a graph $\Gamma$ or $\cG$ is the number of independent cycles in the graph; equivalently, it is given by $\beta = E - V + 1$. We call $\Gamma$ a \emph{tree} if it contains no cycles, i.e., between each two points on the graph there is a unique non-self-intersecting path connecting them; equivalently, if $\beta = 0$. Motivated by the situation on trees, see \cite{KRpamm},
% and in accordance with common usage,
we define the \emph{boundary} $\partial \Gamma$ to be the set of vertices of degree one,
\begin{displaymath}
 \partial \Gamma := \{ v \in \cV: \deg v = 1 \};
\end{displaymath}
we point out that the notion of the boundary of a graph needs to be interpreted with care, see the discussion in the introduction. Note that we view $\partial \Gamma$ as a subset of the metric space $\Gamma$ but identify it with a part of the discrete graph $\cG$; when this perspective is appropriate we will write $\partial \cG$ instead of $\partial \Gamma$.  We call an edge $e$ a {\em bridge} if $\Gamma$ is no longer connected after removing $e$ (but not its endpoint vertices) from the graph. In particular, in a tree every edge is a bridge. Finally, we note the following property for future reference.

\begin{definition}
\label{def:generic}
We say that a property $P$ holds \emph{generically} if, for every discrete graph $\cG$ which is not a cycle graph (see Example~\ref{ex:cycle}), the set of all vectors $y \in \R_+^E$ for which $P$ holds for the associated metric graph with edge lengths $y$ is residual, i.e., of the second Baire category (that is, a countable intersection of open dense sets).
\end{definition}

In particular, if a property holds generically, then given any graph $\Gamma$ it is possible to perturb the edge lengths by an arbitrarily small amount, such that the property holds on the perturbed graph.

\subsection{Function spaces and the standard Laplacian}

If $f:\Gamma \to \C$ or $\R$ is any function defined on $\Gamma$, then we write $f_e :=f |_e$ for the restriction of $f$ to a given edge $e \in \cE$. We use the standard function spaces on $\Gamma$, namely:
\begin{enumerate}
\item[(1)] the space of square-integrable functions $L^2 (\Gamma) = \oplus_{e \in \cE} L^2 (0, L (e))$;
\item[(2)] the space $C (\Gamma)$ of functions being continuous on each edge and having, at each vertex, a value independent of the choice of the incident edge; and
\item[(3)] the Sobolev spaces
\begin{displaymath}
	\widetilde{H}^k (\Gamma) := \big\{ f \in L^2 (\Gamma): f_e \in H^k(0, L (e)) \text{ for all } e \in \cE \big\}, \quad k = 1, 2, \dots,
\end{displaymath}
and $H^1 (\Gamma) = \widetilde H^1 (\Gamma) \cap C (\Gamma)$; cf.~\cite[Section~1.4]{BK13} or \cite[Section~2]{BKKM18}.
\end{enumerate}
Since $\Gamma$ has finite total length, we may identify any continuous function with a function in $L^2 (\Gamma)$ and the embedding $C(\Gamma) \hookrightarrow L^2 (\Gamma)$ of $C (\Gamma)$, equipped with the supremum norm, in $L^2 (\Gamma)$ is continuous. Likewise, the embedding $H^1 (\Gamma) \hookrightarrow C(\Gamma)$ is continuous.

We consider the Laplacian in $L^2 (\Gamma)$ with standard (continuity--Kirchhoff) vertex conditions given by
\begin{align*}
 (- \Delta_\Gamma f)_e & = - f_e'' \qquad \text{for all}~e \in \cE, \\
 \dom (- \Delta_\Gamma) & = \left\{ f \in H^1 (\Gamma) \cap \widetilde H^2 (\Gamma) :
	\partial_\nu f (v) = 0~\text{for~all}~v \in \cV \right\},
\end{align*}
where 
\begin{align*}
 \partial_\nu f (v) = \sum_{t (e) = v} f_e' (L (e)) - \sum_{o (e) = v} f_e' (0)
\end{align*}
is the sum of derivatives (taken in the direction pointing towards the vertex) at $v$ on all edges incident to $v$. This operator is associated with the usual sesquilinear form given by
\begin{displaymath}
    \int_\Gamma f' \overline{g'}\,\textrm{d}x \equiv \sum_{e \in \cE} \int_e f_e' \overline{g_e'}\,\textrm{d}x
\end{displaymath}
defined for $f, g \in H^1 (\Gamma)$ in the sense of, e.g.,~\cite[Chapter~VI]{K95}. Note that up to unitary equivalence this operator is independent of the choice of orientation of the edges, i.e., the choice of originating and terminal vertices. Its spectrum consists of a discrete set of real eigenvalues, each of finite multiplicity; we count them with their multiplicities and denote them by
\begin{align*}
 0 = \mu_1 (\Gamma) < \mu_2 (\Gamma) \leq \mu_3 (\Gamma) \leq \dots.
\end{align*}
We may choose the corresponding eigenfunctions to be real-valued -- which we will \emph{always} do, here and throughout, without further comment -- and form an orthonormal basis of $L^2 (\Gamma)$. Each of these eigenfunctions is in $H^1 (\Gamma)$ and thus, in particular, continuous and defined everywhere on $\Gamma$. The smallest nontrivial eigenvalue $\mu_2 (\Gamma)$ (which we stress may be multiple) admits the variational characterisation
\begin{displaymath}
	\mu_2 (\Gamma) = \inf \left\{ \frac{\int_\Gamma |f'|^2\,\textrm{d}x}{\int_\Gamma |f|^2\,\textrm{d}x}: 0 \neq f \in H^1(\Gamma)
	\text{ and } \int_\Gamma f\,\textrm{d}x = 0 \right\},
\end{displaymath}
with the infimum being achieved exactly by the functions in the eigenspace of the value $\lambda = \mu_2 (\Gamma)$. We stress that if $\mu_2 (\Gamma)$ is multiple, then we consider \emph{all} eigenfunctions in its (more than one-dimensional) eigenspace. For more details on the properties of Laplace-type operators on metric graphs, we refer to Section~1.4 and Chapter~3 of \cite{BK13}. We finish by noting that the operator $- \Delta_\Gamma$ does not feel vertices of degree two as due to the continuity and Kirchhoff conditions one may replace two edges $e, \hat e$ connected by a vertex of degree two by one edge of length $L (e) + L (\hat e)$ without changing the operator (up to unitary equivalence); see also \cite[Section~3]{BKKM17}.

\subsection{The hot spots of a metric graph}

As mentioned in the introduction, given a metric graph $\Gamma$, we will be interested in the sets of (local and global) extrema of the eigenfunctions corresponding to $\mu_2 (\Gamma)$, which we will refer to as hot spots. We recall that we will always restrict to real-valued eigenfunctions, and that a basis of the eigenspace that only consists of real-valued functions can always be chosen. We denote the set of global hot spots by
\begin{displaymath}
	M := \Big\{ x \in \Gamma: \text{ there exists $\psi$ corresp.~to $\mu_2 (\Gamma)$ such that } %\\
	\psi (x) = \max_{y \in \Gamma} \psi (y) \Big\}
\end{displaymath}
%\end{multline*}
and its local version by
\begin{multline*}
	M_{\loc} := \Big\{ x \in \Gamma: \text{ there exists $\psi$ corresp.~to $\mu_2 (\Gamma)$ such that } \\
	0 \neq \psi (x) = \max_{y \in B_\varepsilon (x)} \psi (y) \text{ for some } \varepsilon > 0 \Big\},
\end{multline*}
where $B_\varepsilon (x) = \{ y \in \Gamma: \dist (x,y) < \varepsilon \}$ denotes the open $\varepsilon$-ball centred at $x$. (Since $\psi$ is an eigenfunction if and only if $-\psi$ is, we could equally replace the maximum with the minimum in the above definitions.) If we wish to emphasise the dependence of $M$ and $M_{\loc}$ on the graph $\Gamma$, then we may alternatively write $M(\Gamma)$ and $M_{\loc}(\Gamma)$, respectively. We first note the following sign property of the local extrema.

\begin{lemma}\label{lem:maxPos}
Let $\psi$ be any eigenfunction of $- \Delta_\Gamma$. If $\psi$ takes a nonzero local maximum (or minimum, respectively) at some point $x_0 \in \Gamma$ then $\psi (x_0) > 0$ (or $\psi (x_0) < 0$, respectively).
\end{lemma}

\begin{proof}
Assume that $\psi$ has a local maximum at $x_0$ with $\psi (x_0) < 0$; without loss of generality $x_0$ is a vertex as we may interpret any inner point of an edge as a vertex of degree two. Let $e$ be an edge incident to $x_0$. Then $\psi_e$ is differentiable and satisfies the eigenvalue equation up to $x_0$, that is, $\psi_e'' (x_0) = - \mu_2 (\Gamma) \psi_e (x_0) > 0$ by assumption, that is, $\psi_e$ is convex close to $x_0$, which contradicts the fact that $\psi$ takes a local maximum at $x_0$. The statement on minima follows by replacing $\psi$ by $- \psi$.
\end{proof}

For any $\Gamma$, it is immediate from the definitions and the fact that any eigenfunction for $\mu_2(\Gamma)$ must change sign in $\Gamma$ that
\begin{equation*}
% \label{eq:localglobal}
	\emptyset \neq M \subset M_{\loc}.
\end{equation*}

Finally, for a fixed eigenfunction $\psi$ (not necessarily corresponding to $\mu_2(\Gamma)$, although in practice we will only be interested in this case), we will denote its set of global and local extrema by
\begin{align*}
	M_\psi := \Big\{ x \in \Gamma: \psi (x) = \max_{y \in \Gamma} \psi (y) \text{ or } \psi (x) = \min_{y \in \Gamma} \psi (y)\Big\}
\end{align*}
and
\begin{multline*}
	M_{\psi,\loc} := \Big\{ x \in \Gamma: 0 \neq \psi (x) = \max_{y \in B_\varepsilon (x)} \psi (y) \text{ for some } \varepsilon > 0\\
	\text{or } 0 \neq \psi (x) = \min_{y \in B_\varepsilon (x)} \psi (y) \text{ for some } \varepsilon > 0 \Big\},
\end{multline*}
respectively. Again, for any $\psi$, we clearly have
\begin{equation*}
% \label{eq:localglobal-psi}
	\emptyset \neq M_\psi \subset M_{\psi,\loc},
\end{equation*}
as well as
\begin{displaymath}
	M = \bigcup_{\psi} M_\psi, \qquad M_{\loc} = \bigcup_\psi M_{\psi,\loc},
\end{displaymath}
where in both cases the union is taken over all eigenfunctions $\psi$ corresponding to the same eigenvalue $\mu_2(\Gamma)$. Note that due to the sinusoidal form of the eigenfunctions, the sets $M_\psi$ and $M_{\psi,\loc}$ are always finite. Observe, however, that the analogous statement for the sets $M$ and $M_\loc$ may be false as we discuss below.

\section{Hot spots of special graphs}
\label{sec:examples}

In this section we introduce a few special classes of graphs and discuss their hot spots. This may give the reader an impression what the sets of hot spots defined above may look like in specific examples. Moreover, these graphs will play a role in the forthcoming considerations. 

\begin{example}\label{ex:path}
A {\em path graph} is a connected graph consisting of two vertices of degree one and an arbitrary number of vertices of degree two, see the left-hand side of Figure~\ref{fig:pathCycle}.
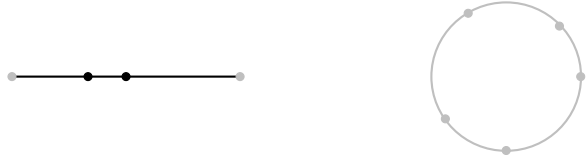
\begin{figure}[h]
\begin{tikzpicture}
\coordinate (a) at (-1,0);
\coordinate (b) at (0,0);
\coordinate (c) at (0.5,0);
\coordinate (d) at (2,0);
\coordinate (e) at (5.5,0);
\draw[thick] (a) -- (b);
\draw[thick] (b) -- (c);
\draw[thick] (c) -- (d);
\draw[lightgray,fill] (a) circle (1.5pt);
\draw[fill] (b) circle (1.5pt);
\draw[fill] (c) circle (1.5pt);
\draw[lightgray,fill] (d) circle (1.5pt);
\draw[lightgray,thick] (e) circle (28pt);
\draw[lightgray,fill] (6.48,0) circle (1.5pt);
\draw[lightgray,fill] (6.2,0.67) circle (1.5pt);
\draw[lightgray,fill] (5.0,0.84) circle (1.5pt);
\draw[lightgray,fill] (5.5,-0.98) circle (1.5pt);
\draw[lightgray,fill] (4.7,-0.56) circle (1.5pt);
\end{tikzpicture}
\caption{Left: a path graph with the hot spots marked in grey. Right: a cycle graph; here every point is a hot spot.}
\label{fig:pathCycle}
\end{figure}
As vertices of degree two do not influence the Laplacian with standard vertex conditions, for our purposes we may identify any metric path graph $\Gamma$ with the interval $[0, L (\Gamma)]$ and the Laplacian on $\Gamma$ with the Laplacian on that interval with Neumann boundary conditions $\psi' (0) = \psi' (L (\Gamma)) = 0$. In particular, we have $\mu_2 (\Gamma) = \frac{\pi^2}{L (\Gamma)^2}$ with corresponding eigenfunction $\cos \frac{\pi x}{L (\Gamma)}$. Consequently, $M = M_\loc = \partial \Gamma$ in this case.
\end{example}

\begin{example}
\label{ex:cycle}
A {\em cycle graph} (or just {\em cycle}) is a connected graph for which each vertex has degree two, see the right-hand side of Figure~\ref{fig:pathCycle}. The Laplacian on such a graph can be identified with the Laplacian on $[0, L (\Gamma)]$ with periodic boundary conditions 
\begin{align*}
 \psi (0) & = \psi (L (\Gamma)), \\ 
 \psi' (0) & = \psi' (L (\Gamma)),
\end{align*}
and the first positive eigenvalue equals $\mu_2 (\Gamma) = \frac{4 \pi^2}{L (\Gamma)^2}$. The corresponding eigenspace is two-dimensional and consists of the functions
\begin{align*}
 \psi(x) = A \cos \Big(\frac{2 \pi x}{L (\Gamma)} - c \Big), \qquad A, c \in \R.
\end{align*}
In particular, by choosing the shift $c$ appropriately we may move the (unique) local maximum of $\psi$ to any point on the cycle. As a consequence, on each cycle graph we have $M = M_\loc = \Gamma$, even though $M_\psi = M_{\psi,\loc}$ always contains exactly two points.
\end{example}

\begin{example}
\label{ex:pumpkin}
A {\em pumpkin graph} consists of two vertices $v_1, v_2$ and a collection of edges each of which connects $v_1$ with $v_2$, see the left-hand side of Figure~\ref{fig:pumpkinStar}. If $\Gamma$ is equilateral with $L (e) = 1$ for each edge $e$ then $\mu_2 (\Gamma) = \pi^2$. The corresponding eigenspace is $E$-dimensional and is spanned by $E - 1$ functions each of which has support on two edges and vanishes at both vertices, and one function equal to $\cos (\pi x)$ on each edge (assuming that all edges are parametrised in the same direction). By taking linear combinations of these eigenfunctions it can be seen that any point on the graph is a hot spot, i.e.\ $M = M_\loc = \Gamma$; see Proposition~\ref{prop:m-gamma-examples} below.
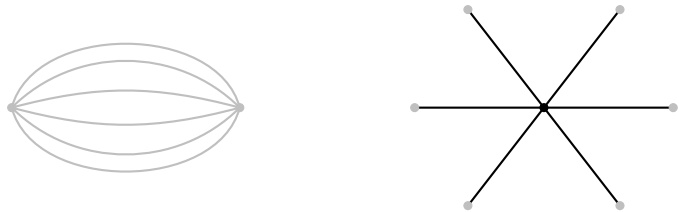
\begin{figure}[h]
\begin{tikzpicture}
% Pumpkin
\coordinate (a) at (0,0);
\coordinate (b) at (3,0);
\draw[lightgray, fill] (0,0) circle (1.5pt);
\draw[lightgray, fill] (3,0) circle (1.5pt);
\draw[lightgray, thick, bend left=15] (a) edge (b);
\draw[lightgray, thick, bend left=45] (a) edge (b);
\draw[lightgray, thick, bend left=75] (a) edge (b);
\draw[lightgray, thick, bend right=15] (a) edge (b);
\draw[lightgray, thick, bend right=45] (a) edge (b);
\draw[lightgray, thick, bend right=75] (a) edge (b);
% Star
\coordinate (c) at (7,0);
\coordinate (d) at (5.3,0);
\coordinate (e) at (8.7,0);
\coordinate (f) at (6,1.3);
\coordinate (g) at (6,-1.3);
\coordinate (h) at (8,1.3);
\coordinate (i) at (8,-1.3);
\draw[thick] (c) edge (d);
\draw[thick] (c) edge (e);
\draw[thick] (c) edge (f);
\draw[thick] (c) edge (g);
\draw[thick] (c) edge (h);
\draw[thick] (c) edge (i);
\draw[fill] (c) circle (1.5pt);
\draw[lightgray, fill] (d) circle (1.5pt);
\draw[lightgray, fill] (e) circle (1.5pt);
\draw[lightgray, fill] (f) circle (1.5pt);
\draw[lightgray, fill] (g) circle (1.5pt);
\draw[lightgray, fill] (h) circle (1.5pt);
\draw[lightgray, fill] (i) circle (1.5pt);
\end{tikzpicture}
\caption{A pumpkin and a star graph. The set $M (\Gamma)$ in the equilateral case is marked in grey.}
\label{fig:pumpkinStar}
\end{figure}
If $\Gamma$ is not equilateral and $E \geq 3$ then we may (still) find a basis of eigenfunctions each of which is either reflection or rotation symmetric with respect to the midpoints of each edge. Assume that $\Gamma$ has a unique longest edge and $\mu_2 (\Gamma)$ is simple with eigenfunction not vanishing identically on any edge (as is the case generically \cite{BeLi17}); then as the eigenfunction is either reflection or rotation symmetric, $M_{\loc}$ must correspondingly consist either of the set of midpoints of each edge (the reflection symmetric case), \emph{or} the two points on the longest edge at distance $\pi/\sqrt{\mu_2(\Gamma)}$ from each other and equidistant from the midpoint of the edge (the rotationally symmetric case).
\end{example}

\begin{example}
A {\em star graph} is a graph with a ``star vertex'' $v_0$ of degree $E$ and $E$ vertices of degree one, see the right-hand side of Figure~\ref{fig:pumpkinStar}. In the equilateral case with $L (e) = 1$ for each $e$ the lowest positive eigenvalue equals $\mu_2 (\Gamma) = \pi^2 / 4$. The corresponding eigenspace is $E - 1$-dimensional and is spanned by functions each being supported on a pair of two edges; each of these basis functions has its minimum at one boundary vertex and its maximum at another. Thus $M = M_\loc = \partial \Gamma$. For non-equilateral star graphs we will see in Corollary~\ref{cor:treelocalhotspots} below that $M \subset M_\loc \subset \partial \Gamma$, but equality of these sets does not necessarily hold.
\end{example}

\begin{example}
\label{ex:flower}
A {\em flower graph} consists of one vertex and a number of loops (``petals'') attached to this vertex; cf.\ the left-hand side of Figure~\ref{fig:flowerComplete}.
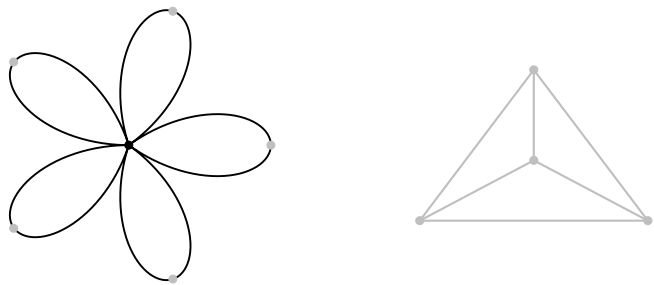
\begin{figure}[h]
\begin{center}
% flower
\begin{minipage}[c]{5cm}
\begin{tikzpicture}[scale=0.6]
\begin{polaraxis}[grid=none, axis lines=none]
\addplot[very thick,mark=none,domain=0:360,samples=300] { abs(cos(5*x/2))};
\draw[fill] (0,0) circle (2.5pt);
\draw[lightgray, fill] (100,0) circle (2.5pt);
\draw[lightgray, fill] (31,95) circle (2.5pt);
\draw[lightgray, fill] (-81,59) circle (2.5pt);
\draw[lightgray, fill] (-81,-59) circle (2.5pt);
\draw[lightgray, fill] (31,-95) circle (2.5pt);
\end{polaraxis}
\end{tikzpicture}
\end{minipage} \qquad
% complete graph
\begin{minipage}{3.75cm}
\begin{tikzpicture}
\coordinate (c) at (6,-1);
\coordinate (d) at (9,-1);
\coordinate (e) at (7.5,-0.2);
\coordinate (f) at (7.5,1);
\draw[lightgray, fill] (c) circle (1.5pt);
\draw[lightgray, fill] (d) circle (1.5pt);
\draw[lightgray, fill] (e) circle (1.5pt);
\draw[lightgray, fill] (f) circle (1.5pt);
\draw[lightgray, thick] (c) edge (d);
\draw[lightgray, thick] (c) edge (e);
\draw[lightgray, thick] (c) edge (f);
\draw[lightgray, thick] (d) edge (e);
\draw[lightgray, thick] (d) edge (f);
\draw[lightgray, thick] (e) edge (f);
\end{tikzpicture}
\end{minipage}
\end{center}
\caption{A flower graph and a complete graph. The hot spots for the equilateral case are marked in grey.}
\label{fig:flowerComplete}
\end{figure}
A flower graph with exactly two petals is called a {\em figure-8 graph}. For the equilateral flower with $E \geq 2$ petals of length one each we have $\mu_2 (\Gamma) = \pi^2$ with the corresponding eigenfunctions equal to $\pm \sin (\pi x)$ or constantly zero on each edge. On an arbitrary, not necessarily equilateral flower graph with $E \geq 2$ one has $\mu_2 (\Gamma) < \frac{4 \pi^2}{L (e)^2}$ for each edge $e$ (this follows from the strict inequality statement of \cite[Theorem~3.10(2)]{BKKM18}, since $\Gamma$ can be formed by attaching the other $E-1$ edges as a pendant to the loop $e$, which has an eigenfunction which does not vanish at the point of attachment). Thus the continuity condition $\psi_e (0) = \psi_e (L (e))$ implies that the restriction of any corresponding eigenfunction to any edge $e$ is (reflection) symmetric with respect to the midpoint of the edge. Thus the midpoint is the only critical point on the edge and $M_\loc$ (and hence $M$) is contained in the set of edge midpoints on any flower graph.
\end{example}

\begin{example}
\label{ex:complete}
A {\em complete graph} is a graph such that for each pair of distinct vertices there is exactly one edge connecting the two and every vertex has degree $V - 1$, see the right-hand side of Figure~\ref{fig:flowerComplete}. In the equilateral case with $L (e) = 1$ for each $e \in E$ we will show in Proposition~\ref{prop:m-gamma-examples} that $M = M_\loc = \Gamma$ as long as $V \geq 3$ (see also Lemma~\ref{lem:complete-ef} for a description of those eigenfunctions associated with $\mu_2(\Gamma)$ which will be of principal importance for us).
\end{example}

\begin{example}
\label{ex:lasso}
A {\em lasso graph}, also known as a {\em lollipop} or a {\em tadpole graph}, is a graph consisting of a loop and a pendant edge attached to each other, see Figure~\ref{fig:lolly}.
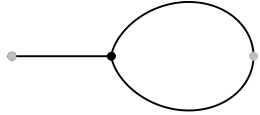
\begin{figure}[h]
\vspace*{-10mm}
\begin{tikzpicture}[scale=0.6]
\begin{polaraxis}[grid=none, axis lines=none]
\addplot[very thick,mark=none,domain=-90:90,samples=300] {cos(x)^2};
\draw[fill] (0,0) circle (2.5pt);
\draw[fill] (-700,0) circle (2.5pt);
\draw[very thick] (0,0) edge (-700,0);
\draw[lightgray, fill] (-700,0) circle (2.5pt);
\draw[lightgray, fill] (1000,0) circle (2.5pt);
\end{polaraxis}
\end{tikzpicture}\vspace*{-10mm}
\caption{A lasso graph with its hot spots in grey.}
\label{fig:lolly}
\end{figure}
Such a graph is a special case of a so-called {\em pumpkin chain}; the lowest eigenvalue of the standard Laplacian is simple and, since each of the two ``pumpkins'' in the chain is equilateral, the corresponding eigenfunction is monotonic along the chain and takes its maximum and minimum at the ``end points'' of the chain only, i.e.\ on the vertex of degree one and the midpoint of the loop; see~\cite[Section~5]{BKKM18} (in particular Lemma~5.5 there) for more details.
\end{example}

\section{Summary of results}
\label{sec:main}

We will now present the results of this article. Their proofs will be given in later sections. Throughout the paper, we will \emph{always} make the following assumption on our graph $\Gamma$, without exception. Thus, whenever we speak of ``a graph $\Gamma$'' or ``any graph $\Gamma$'', we always mean graphs $\Gamma$ satisfying this assumption.

\begin{assumption}\label{ass}
The metric graph $\Gamma$ has a finite set of vertices and a finite set of edges, and the length of every edge is finite; in particular, $\Gamma$ is a compact metric space. Moreover, $\Gamma$ is connected.
\end{assumption}

\subsection{On the number of hot spots}
\label{sec:number-summary}

We first wish to consider the number of hot spots that a graph can have, since the set of these distinguished points may be very large. Indeed, as already claimed in Section~\ref{sec:examples}, one may find examples $\Gamma$ other than cycles, for which $M = \Gamma$.

\begin{proposition}
\label{prop:m-gamma-examples}
We have that $M=M_{\loc}=\Gamma$ whenever $\Gamma$ is an equilateral pumpkin (cf.\ Example~\ref{ex:pumpkin}) or an equilateral complete graph with $V \geq 3$ (cf.\ Example~\ref{ex:complete}).
\end{proposition}

This proposition says that in a certain sense all points on any equilateral pumpkin or equilateral complete graph are equally connected. However, these graphs are easily seen to be exceptional in the sense that for a ``generic'' graph, $M$ is finite. For future reference, let us explain this in more detail: since $\mu_2 (\Gamma)$ is generically simple \cite{F05} (cf.\ Definition~\ref{def:generic}), it follows immediately that $M$, in fact $M_{\loc}$, is generically finite. However, we can say more. Firstly, in the case of trees, $M_{\loc}$ is \emph{always} finite. 

\begin{proposition}
\label{prop:tree-m-finite}
Suppose $\Gamma$ is a tree. Then $M_{\loc}$, and hence also $M$, are finite sets.
\end{proposition}

Moreover, it is not difficult to show that if two points of $\Gamma$ sufficiently close to each other are in $M_{\loc}$, then so are all points in between them. This leads to the following dichotomy.

\begin{proposition}
\label{prop:m-finite-or-uncountable}
The set $M_{\loc}$ always has a finite number of connected components; in particular, it is either finite or uncountable.
\end{proposition}

We expect that Proposition~\ref{prop:m-finite-or-uncountable} actually remains true for $M$ in place of $M_{\loc}$; see Conjecture~\ref{conj:m-size}. We next give a stronger version of the statement that $M$ is generically finite: the next theorem states that, up to a small perturbation of the graph which does not essentially change the topology, the idea that each graph should have one maximum (hottest spot) and one minimum (coldest spot) is correct. More precisely, the unperturbed graph has the same underlying discrete graph as the perturbed graph, if we allow that some of the edge lengths of the former may be zero. Here and in what follows we use $|\cdot |$ to denote the cardinality of a set.

\begin{theorem}
\label{thm:m-generically-2}
Given $\Gamma$, for each $\eps > 0$ there exists a graph $\Gamma_\eps$ obtained from $\Gamma$ by modifying the length of each edge by less than $\eps$ and possibly attaching finitely many pendant edges of length less than $\eps$ to points in 
% the doubly connected part of 
$\Gamma$ such that $\mu_2 (\Gamma_\eps)$ is simple, $M (\Gamma_\eps) \subset \partial \Gamma_\eps$ holds, and the corresponding eigenfunction has exactly one minimum and one maximum, i.e., $|M (\Gamma_\eps)| = 2$.
\end{theorem}

We expect that Theorem~\ref{thm:m-generically-2} can be sharpened in the sense that the property $|M| = 2$ holds generically, that is, one may avoid attaching additional short pendant edges; see Conjecture~\ref{conj:Mgeneric} below.

For completeness' sake, we observe explicitly that $M$ can take on any finite size.

\begin{proposition}
\label{prop:number}
For any $n \geq 2$, there is a graph $\Gamma$ for which $|M| = |M_{\loc}| = n$.
\end{proposition}

We will give the proofs of Propositions~\ref{prop:m-gamma-examples} and~\ref{prop:m-finite-or-uncountable} in Section~\ref{sec:number-uncountable}, the proof of the key Theorem~\ref{thm:m-generically-2} is the subject of Section~\ref{sec:number-generic}, and an explicit construction that proves Proposition~\ref{prop:number} will be given in Section~\ref{sec:number-examples}. The proof of Proposition~\ref{prop:tree-m-finite} is based on results on the location of hot spots presented in the following section; in fact, it is an immediate consequence of Corollary~\ref{cor:treelocalhotspots} below.

\subsection{On the location of the hot spots}
\label{sec:location-summary}

In this section we study the location of hot spots on a graph and, in particular, their relation to the boundary and to the diameter of the graph. Moreover, we investigate in which regions of the graph hot spots may or may not be located.

First, we give a negative answer to the question of whether the hot spots of a metric graph need to be located on the boundary. Trivially, this cannot be true for graphs with empty boundary such as, e.g., cycles, pumpkins or complete graphs. However, the following example shows that this can fail also for graphs with nonempty boundary.

\begin{example}\label{ex:verrueckt}
Consider a ``figure-8 with a small perturbation'', more precisely, a graph $\Gamma$ with a central vertex $v_0$ and four edges attached to $v_0$, two of them loops and two edges connecting $v_0$ to a vertex of degree one each, see Figure~\ref{fig:verrueckt}. 
\begin{figure}[h]
\vspace*{-10mm}
\begin{tikzpicture}[scale=0.6]
\begin{polaraxis}[grid=none, axis lines=none]
\addplot[very thick,mark=none,domain=0:360,samples=300] {cos(x)^2};
\draw[fill] (0,0) circle (2.5pt);
\draw[fill] (0,30) circle (2.5pt);
\draw[fill] (0,-30) circle (2.5pt);
\draw[very thick] (0,0) edge (0,30);
\draw[very thick] (0,0) edge (0,-30);
\draw[fill,lightgray] (100,0) circle (2.5pt);
\draw[fill,lightgray] (-100,0) circle (2.5pt);
\end{polaraxis}
\end{tikzpicture}\vspace*{-10mm}
\caption{The perturbed ``figure-8'' graph of Example~\ref{ex:verrueckt} and its hot spots in grey.}
\label{fig:verrueckt}
\end{figure}
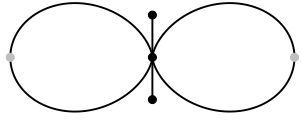
Assume that each loop has length $\pi$ and each of the ``boundary edges'' has length $\varepsilon$ for some $\varepsilon \geq 0$. In the case $\varepsilon = 0$ this is just an equilateral figure-8 with $\mu_2 (\Gamma) = 1$ (with multiplicity one); the corresponding eigenfunction has a zero at $v_0$. As $\mu_3 (\Gamma)$ is strictly larger and the eigenvalues depend continuously on $\varepsilon$ (see, e.g., the appendix), letting $\varepsilon$ grow from zero to a sufficiently small positive length will result in a graph whose third eigenvalue continues to be strictly larger than $1$. On the other hand, $1$ continues to be an eigenvalue (with the same eigenfunction, extended by zero to the boundary edges), and it follows that the second eigenvalue still equals $1$, has multiplicity one and the corresponding eigenfunction vanishes on the boundary edges. In particular, $M \cap \partial \Gamma = M_{\loc} \cap \partial \Gamma = \emptyset$.
\end{example}

By a similar perturbation-type argument one may, more generally, add boundary to any given graph without increasing $M$. In particular, one may construct graphs $\Gamma$ with an arbitrarily large number of boundary vertices such that $M \cap \partial \Gamma = M_\loc \cap \partial \Gamma = \emptyset$ holds.

On the other hand, on a graph whose hot spots lie on the boundary it is easily possible to introduce a geometric perturbation that removes the boundary but changes the position of the hot spots only slightly. A simple but prototypical example is as follows.

\begin{example}\label{ex:undUmgekehrt}
Consider the path graph $\Gamma$ given by the interval $[0, \pi]$; cf.\ Example~\ref{ex:path}. Its hot spots lie on the two boundary vertices, i.e.\ the two endpoints of the interval. However, by attaching arbitrarily small loops to the endpoints, see Figure~\ref{fig:dumbbell}, we obtain a graph with empty boundary, whose hot spots are still at the two points farthest apart from each other, namely the midpoints of the two loops. 
\begin{figure}[h]
\vspace*{-1mm}
\begin{tikzpicture}[scale=0.6]
\draw[fill] (5,0) circle (2.5pt);
\draw[fill] (-5,0) circle (2.5pt);
\draw[very thick] (-5,0) edge (5,0);
\draw[thick] (5.52,0) circle (15pt);
\draw[thick] (-5.52,0) circle (15pt);
\draw[fill,lightgray] (6.04,0) circle (2.5pt);
\draw[fill,lightgray] (-6.04,0) circle (2.5pt);
\end{tikzpicture}\vspace*{-1mm}
\caption{The perturbed path graph of Example~\ref{ex:undUmgekehrt} and its hot spots in grey.}
\label{fig:dumbbell}
\end{figure}
In fact, the resulting graph is a special case of a so-called {\em locally equilateral pumpkin chain}, see~\cite[Section~5.1]{BKKM18}, and the corresponding eigenfunction corresponding to $\mu_2$, unique up to scalar multiples, is a monotonic function of the distance to either of the grey points {\em (monotonic along the chain)} \cite[Lemma~5.1]{BKKM18}.
\end{example}

Clearly, the principle of attaching a loop is more general, but we abstain here from delving into this here.

Next we provide a theorem that excludes hot spots from certain regions of a graph, more precisely from bridges. In the following we denote by $\cD_\Gamma \subset \Gamma$ the \emph{doubly connected part} of $\Gamma$, i.e., the closed subgraph consisting of all points that are part of a cycle in $\Gamma$; note that $\cD_\Gamma$ can be obtained by successively removing all edges incident to a vertex of degree one as well as all (further) bridges; see also \cite[Section~6]{BKKM18}. We call $\Gamma$ \emph{doubly connected} if $\Gamma = \cD_\Gamma$. Moreover, we denote by 
\begin{align*}
 \interior \cD_\Gamma = \Gamma \setminus \overline{ \big(\Gamma \setminus \overline{\cD_\Gamma} \big) }
\end{align*}
the interior of the doubly connected part (with respect to the natural metric on $\Gamma$).

\begin{theorem}\label{thm:doublyConnectedPart}
Given any graph $\Gamma$, we have
\begin{align*}
 M \subset M_\loc \subset \partial \Gamma \cup \interior \cD_\Gamma.
\end{align*}
\end{theorem}

For tree graphs the doubly connected part is empty and thus Theorem~\ref{thm:doublyConnectedPart} implies the following statement; a slightly weaker result was given by the authors in~\cite{KRpamm}.

\begin{corollary}
\label{cor:treelocalhotspots}
For any tree $\Gamma$, $M \subset M_{\loc} \subset \partial \Gamma$.
\end{corollary}

We point out that Proposition~\ref{prop:tree-m-finite} is a trivial consequence of Corollary~\ref{cor:treelocalhotspots} as $\partial \Gamma$ only consists of finitely many points.

Finally we investigate to what extent the ``coldest'' and ``hottest'' points on a metric graph are necessarily far apart from each other. In fact this is not necessarily the case:

\begin{proposition}\label{prop:hotspotsClose}
For each $\eps > 0$ there exists a metric tree $\Gamma$ with $\diam \Gamma = 1$ such that $\mu_2 (\Gamma)$ is simple and 
\begin{align*}
 \max \{ \dist (x, y) : x, y \in M \} < \eps.
\end{align*}
\end{proposition}

Although in general the distance between the hottest and coldest spots on a tree does not need to realise the diameter, this is true for the class of star graphs, as the following proposition shows. A corresponding result also holds for flowers; see Remark~\ref{rem:flowerDiameter}, and it can be extended to more general ``star-like'' graphs.

\begin{proposition}\label{prop:starDiameter}
Let $\Gamma$ be a star graph, let $\psi \in \ker (- \Delta_\Gamma - \mu_2 (\Gamma))$ be nontrivial and let $x, y \in \Gamma$ such that $\psi$ takes its global maximum at $x$ and its global minimum at $y$. Then
\begin{align*}
 \dist (x, y) = \diam \Gamma.
\end{align*}
\end{proposition}

We remark that $x$ and $y$ as in the proposition have to lie on $\partial \Gamma$ by Corollary~\ref{cor:treelocalhotspots}.

Proposition~\ref{prop:hotspotsClose} is proved by means of an explicit construction in Example~\ref{ex:krpamm}. The proof of Proposition~\ref{prop:starDiameter} is given in Section~\ref{sec:location-diameter}, while Theorem~\ref{thm:doublyConnectedPart} will turn out to be a corollary of the slightly more general Theorem~\ref{thm:noDisconnect} given below.

\subsection{Graph topology and hot spots}
\label{sec:topology-summary}

We just saw that the global extrema either lie on $\partial\Gamma$ or in the doubly connected part of $\Gamma$. Here we show that the question of which of the two it is can depend essentially on the edge lengths of $\Gamma$, not on its topology, if these are chosen correctly.

\begin{theorem}
\label{thm:examplesummary}
Let $\cG$ be a finite, connected discrete graph. 
\begin{enumerate}
 \item If $|\partial \cG| \geq 1$ then there exists an associated metric graph $\Gamma$ (see Section~\ref{subsec:basic}) such that the eigenvalue $\mu_2 (\Gamma)$ is simple and the corresponding eigenfunction takes its maximum only on $\partial \Gamma$.
 \item If $|\partial \cG| \geq 2$ then there exists an associated metric graph $\Gamma$ such that the eigenvalue $\mu_2 (\Gamma)$ is simple and the corresponding eigenfunction takes its maximum and minimum only on $\partial \Gamma$.
 \item If $\beta \geq 1$ and $\cG$ is not a cycle graph then there exists an associated metric graph $\Gamma$ such that the eigenvalue $\mu_2 (\Gamma)$ is simple and the corresponding eigenfunction takes its maximum only in $\cD_\Gamma$.
 \item If $\beta \geq 2$ then there exists an associated metric graph $\Gamma$ such that the eigenvalue $\mu_2 (\Gamma)$ is simple and the corresponding eigenfunction takes its maximum and minimum only in $\cD_\Gamma$.
\end{enumerate}
\end{theorem}

The proof of this theorem is provided in Section~\ref{sec:topology}.

\begin{remark}
Actually, the proof of Theorem~\ref{thm:examplesummary} will show more than the theorem states: the global extrema may be placed on arbitrary boundary vertices in~(i) and~(ii), inside an arbitrary edge in the doubly connected part in~(iii) and on two arbitrary (but different) cycles in~(iv).
\end{remark}

\section{On the number of hot spots}
\label{sec:number}

\subsection{Graphs with an uncountable number of hot spots: proof of Propositions~\ref{prop:m-gamma-examples} and~\ref{prop:m-finite-or-uncountable}}
\label{sec:number-uncountable}

We start with a technical lemma. Recall that for any two points $x_1, x_2 \in \Gamma$ that lie inside the same edge $e$ we write $\dist_e (x_1, x_2)$ for the distance between $x_1$ and $x_2$ in the metric on the interval $[0, L (e)]$.  Note that $\dist_e (x_1, x_2) \geq \dist (x_1, x_2)$ holds for the distance $\dist (x_1, x_2)$ with respect to the metric on $\Gamma$.

\begin{lemma}
\label{lem:Mgeneration}
Given $\Gamma$, suppose that $x_1,x_2 \in M_{\loc}$ both lie on a given edge $e \in \cE$ and
\begin{displaymath}
	\dist_e (x_1,x_2) \leq \frac{\pi}{2\sqrt{\mu_2(\Gamma)}}.
\end{displaymath}
Then the segment $[x_1,x_2] \subset e$ between $x_1$ and $x_2$ in $e$ lies in $M_{\loc}$. More precisely, for every $y \in [x_1,x_2]$ there exists an eigenfunction whose unique critical point in $[x_1,x_2]$ is a (local) maximum at $y$.
\end{lemma}

\begin{proof}
Identify $e$ with the interval $[0, L]$, where $L = L (e)$, and suppose without loss of generality that $x_1=0$, $x_2=L$. Then since $0 \in M_{\loc}$, there exists an eigenfunction $\psi_0$ which on $e$ up to normalisation has the form $\psi_0 (x) = \cos (k x)$, $x \in [0,L]$, where $k = \sqrt{\mu_2 (\Gamma)}$. Similarly, since the endpoint of $e$ corresponding to $L$ is in $M_{\loc}$ there exists an eigenfunction $\psi_{L} (x) = \cos (k (x-L))$, $x \in [0,L]$.

Now since $L \leq \pi/(2 k)$ by assumption, we have that $\psi_0$ is non-negative and monotonically decreasing on $[0,L]$, while $\psi_{L}$ is non-negative and monotonically increasing on $[0,L]$. We claim that for all $y \in (0,L)$ there exists some constant $\alpha_y>0$ such that the eigenfunction
\begin{align*}
 \psi_y := \psi_0 + \alpha_y \psi_L 
\end{align*}
(which means
\begin{displaymath}
	\psi_y (x) = \cos (k x) + \alpha_y \cos (k (x-L))
\end{displaymath}
for $x \in [0, L]$) reaches its unique maximum on $[0,L]$ at $y$; this will prove the lemma. To prove the claim, we calculate that
\begin{displaymath}
	\psi_y'(x) = [-\sin (k x) - \alpha_y \sin (k(x-L))]k = 0
\end{displaymath}
if and only if
\begin{displaymath}
	\alpha_y = - \frac{\sin (k x)}{\sin (k(x-L))};
\end{displaymath}
note that $\alpha_y > 0$ since the numerator is positive and the denominator is negative, whence $\psi_y > 0$ on $(0,L)$. Hence, given $y$, if we set
\begin{equation*}
% \label{eq:alphay}
	\alpha_y := - \frac{\sin (k y)}{\sin (k(y-L))},
\end{equation*}
then this is the unique value of $\alpha_y$ for which $\psi_y'(y)=0$. Since
\begin{displaymath}
	\psi_y''(x) = - \mu_2 (\Gamma) \psi_y(x) < 0
\end{displaymath}
for all $x \in (0,L)$, $y$ is clearly the unique maximum of $\psi_y$ on $[0,L]$. This proves the claim.
\end{proof}

With the help of this lemma, we can prove that $M=\Gamma$ for the graphs claimed in Proposition~\ref{prop:m-gamma-examples}. First, however, we need a description of the eigenfunctions of the equilateral complete graph.

\begin{lemma}
\label{lem:complete-ef}
Let $\Gamma$ be the equilateral complete graph on $V \geq 4$ vertices (and $V(V-1)/2$ edges of length $1$ each). Fix any vertex $v_0$ and denote by $v_1,\ldots, v_{(V - 2)(V-1)/2}$ the set of points in the middle of any edge not incident to $v_0$ (thus $\dist (v_0,v_k) = \diam \Gamma$ for all $k=1,\ldots, (V - 2) (V-1)/2$, cf.\ Figure~\ref{fig:complete-ef}). Then there exists an eigenfunction $\psi$ for $\mu_2 (\Gamma)$ taking on its global maximum at $v_0$, its global minimum at the $v_k$, and with no other local minima or maxima on $\Gamma$.
\end{lemma}

\begin{figure}
\label{fig:complete-ef}
\begin{tikzpicture}[scale=1.2]
\coordinate (c) at (0,-1);
\coordinate (d) at (3,-1);
\coordinate (e) at (1.5,-0.2);
\coordinate (f) at (1.5,1);
\coordinate (g) at (1.5,-1);
\coordinate (h) at (0.75,-0.6);
\coordinate (i) at (2.25,-0.6);
\draw[thick] (c) edge (d);
\draw[thick] (c) edge (e);
\draw[thick] (c) edge (f);
\draw[thick] (d) edge (e);
\draw[thick] (d) edge (f);
\draw[thick] (e) edge (f);
\draw[fill] (c) circle (1.5pt);
\draw[fill] (d) circle (1.5pt);
\draw[fill] (e) circle (1.5pt);
\draw[lightgray,fill] (f) circle (1.5pt);
\draw[lightgray,fill] (g) circle (1.5pt);
\draw[lightgray,fill] (h) circle (1.5pt);
\draw[lightgray,fill] (i) circle (1.5pt);
\node at (f) [anchor=south west] {$v_0$};
\node at (g) [anchor=north] {$v_1$};
\node at (h) [anchor=south] {$v_2$};
\node at (i) [anchor=south] {$v_3$};
\end{tikzpicture}
\caption{The complete graph admits an eigenfunction for $\mu_2$ whose maximum is at $v_0$ and minimum is achieved at the other $v_k$, with no other critical points.}
\end{figure}
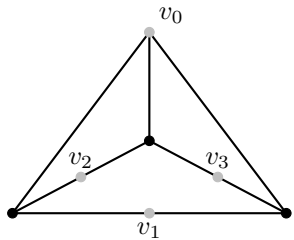

\begin{proof}
Fix $v_0$ and denote by $e_1,\ldots,e_{V-1}$ the edges incident to $v_0$. Then by symmetry of $\Gamma$, there exists an eigenfunction $\psi$ for $\mu_2 (\Gamma)$ which is invariant under permutation of the edges $e_1,\ldots,e_{V-1}$, \emph{and} which is invariant under permutation of any two other edges not incident to $v_0$, which we will denote by $e_V,\ldots,e_E$. (We omit the elementary proof of this claim, which follows from the same reasoning as, for example, \cite[Lemma~5.4(1)]{BKKM18}.)

Now $\psi$ cannot have a local extremum on any of the edges $e_1,\ldots,e_{V-1}$, since if it did, by invariance under permutation it would have (without loss of generality) a local maximum on {\em every} edge $e_1,\ldots,e_{V-1}$, a contradiction to Theorem~\ref{thm:noDisconnect} below. Hence it is monotonic on each of these edges. Symmetry (that is, invariance under permutation) now implies that $\psi$ must have a maximum or a minimum, say a maximum, at $v_0$, which is global on $e_1 \cup \dots \cup e_{V-1}$.

Now the invariance of $\psi$ under permutation of the other edges $e_V,\ldots,e_E$, plus the fact that $\psi$ takes on the same value at all the vertices of $e_1,\ldots,e_{V-1}$ different from $v_0$, means that $\psi$ must be symmetric about the midpoint of each of the edges $e_V,\ldots,e_E$. A further application of Theorem~\ref{thm:noDisconnect} (applied to $\psi$ or $-\psi$, as appropriate) when combined with the symmetry arguments means that $\psi$ can have no local maxima on $e_V,\ldots,e_E$, and it can have at most one local minimum on each of these edges. The only possibility is that $\psi$ reaches its global minimum at the respective midpoints of $e_V,\ldots,e_E$, and is otherwise monotonic.
\end{proof}

\begin{proof}[Proof of Proposition~\ref{prop:m-gamma-examples}]
\emph{Equilateral pumpkin:} Fix any edge $e \in \cE$, which we identify with $[0,1]$. As noted in Example~\ref{ex:pumpkin}, there exists an eigenfunction $\psi_0$ which takes the form $\psi_0 (x) = \cos (\pi x)$, $x \in [0,1]$, both on this and every other edge, and another eigenfunction $\psi_1 (x) = \sin (\pi x)$, $x \in [0,1]$, which is negative on another edge and zero on the rest. Note that $\psi_0$ and $\psi_1$ have their unique global maximum at $0$ and $1/2$, respectively. By Lemma~\ref{lem:Mgeneration} (applicable since $\mu_2 (\Gamma) = \pi^2$, so that $\pi/(2\sqrt{\mu_2 (\Gamma)}) = 1/2 = \dist_e (0,1/2)$), for any $y \in (0,1/2)$ there exists an eigenfunction of the form $\psi_y:=\psi_0 + \alpha_y \psi_1$, $\alpha_y >0$, which has its unique maximum in $[0,1/2]$ at $y$.

We claim that $\psi_y$ actually reaches its global maximum at $y$. On $[1/2,1] \subset e$, we have $\psi_y' = \psi_0' + \alpha_y \psi_1' < 0$, so $y$ is certainly the unique maximum on $e$. Since $\psi_y |_e (y) > \psi_y |_e (0) = \psi_0 |_e (0) = 1$ and $\psi_1 \leq 0$ on all the other edges, meaning that $\psi_y \leq \psi_0 \leq 1$ there, it is the unique global maximum on the equilateral pumpkin.

We have thus shown (still under the identification $e \simeq [0,1]$) that $[0,1/2] \subset M$. Symmetry with respect to the midpoint (equivalently, reversing the parametrisation of the edges) yields $[1/2, 1] \subset M$ as well, that is, $e \subset M$. Since $e$ was arbitrary, we conclude that in this case $M$ is the whole equilateral pumpkin.

\emph{Equilateral complete graph:} Obviously we may assume that $V\geq 4$, since $V=3$ corresponds to a loop. The argument is similar to the one for the pumpkin, but now we use the eigenfunctions described in Lemma~\ref{lem:complete-ef}. More precisely, again let $e \simeq [0,1]$ be any edge, where $0$ corresponds to some vertex $v_0$. Then there exists an eigenfunction $\psi_0$ taking its unique global maximum at $0$ and another eigenfunction $\psi_1$ which has a global maximum at $1/2$ (although this maximum is not unique). Since $\mu_2 (\Gamma) < \pi^2$ (cf.\ \cite[Example~3.3 and Theorem~4.2]{KKMM16}), $\psi_0(1/2), \psi_1 (0) > 0$ and Lemma~\ref{lem:Mgeneration} is applicable. Hence for any $y \in [0,1/2]$ there exists a unique $\alpha_y>0$ such that $\psi_y = \psi_0 + \alpha_y \psi_1$ has its unique maximum in $[0,1/2]$ at $y$.

As before, $\psi_y$ is strictly decreasing on $[1/2,1]$ and so $y$ is the unique maximum on $e$. Now fix another edge $\tilde e$. We distinguish between two cases: (1) $\psi_1$ does not have a maximum on $\tilde e$; (2) $\psi_1$ {\em does} have a maximum on $\tilde e$.

In case (1), by the monotonicity of $\psi_1$ on $\tilde e$ we have that $\psi_1|_{\tilde e} \leq \psi_1(v_0) = \psi_1 |_{\tilde e} (0)$. Since $\psi_0$ reaches its global maximum at $v_0$, $\psi_y|_{\tilde e} = \psi_0|_{\tilde e} + \alpha_y \psi_1|_{\tilde e} \leq \psi_0(v_0) + \alpha_y \psi_1|_e (0) = \psi_y (v_0) < \psi_y |_e (y)$, where the strict inequality follows from the uniqueness statement in Lemma~\ref{lem:Mgeneration}. In case (2), both $\psi_0$ and $\psi_1$ are invariant under permutations of $e$ and $\tilde e$, and hence so is $\psi_y$. In particular, $\psi_y$ reaches a unique maximum on $\tilde e$ which is equal to $\psi_y(y)$.

At any rate, $y$ remains a global maximum of $\psi_y$, whence $y \in M$ and hence $[0,1/2] \subset M$.

Finally, as before, since the orientation of the edge and the choice of edge itself were arbitrary, we conclude that the whole complete graph is contained in $M$.
\end{proof}

\begin{proof}[Proof of Proposition~\ref{prop:m-finite-or-uncountable}]
If $M_{\loc}$ is finite, then there is nothing to prove; so suppose it is infinite. Since $\Gamma$ is assumed to have a finite number of edges it suffices to prove that on any edge $e$ of $\Gamma$ the set of connected components of $M_{\loc} \cap e$ is finite. But since each edge has finite length, this is a direct consequence of Lemma~\ref{lem:Mgeneration}: fix $\varepsilon > 0$ small enough (any $\varepsilon \leq \pi/(2 \sqrt{\mu_2(\Gamma)})$ will do), then whenever $x_1,x_2 \in e$ satisfy $\dist_e (x_1,x_2) < \varepsilon$ and $x_1,x_2 \in M_{\loc}$, it follows that $[x_1,x_2] \subset M_{\loc}$.
\end{proof}

\subsection{Uniqueness of minimum and maximum up to a small perturbation: proof of Theorem~\ref{thm:m-generically-2}}
\label{sec:number-generic}

We show that by an arbitrarily small perturbation of the edge lengths and possibly attaching small pendant edges we can always achieve a metric graph with unique ``coldest and hottest points''. We recall that for a metric graph $\Gamma$ the underlying discrete graph is the unique discrete graph that, together with the length function, constitutes $\Gamma$.

The following lemma will be essential to the proof of Theorem~\ref{thm:m-generically-2}.

\begin{lemma}\label{lem:boundary}
Let $\Gamma$ be different from a path or a cycle, let $k \geq 2$ and denote by $\psi$ an eigenfunction corresponding to $\mu_k (\Gamma)$. Then for each $\eps > 0$ there exists a graph $\Gamma_\eps$ obtained from $\Gamma$ by modifying the length of each edge by less than $\eps$ such that $\mu_k (\Gamma_\eps)$ is simple and the corresponding eigenfunction has pairwise distinct values on the boundary vertices.
\end{lemma}

\begin{proof}
We may assume that, after an arbitrarily small perturbation of the edge lengths, each eigenvalue of $\Gamma$ as well as each eigenvalue of any graph that we may obtain from $\Gamma$ by gluing together any pair of boundary vertices is simple and the corresponding eigenfunction is nonzero on all vertices; since there are only finitely many such graphs, the main result of \cite{BeLi17} guarantees that this is possible. In particular, we may find some $\delta > 0$ such that the distance between any two distinct eigenvalues of index no larger than $k$ on any of these graphs is larger than $\delta$. 

Now let $\eps > 0$ and assume that $v_1, v_2$ are two distinct boundary vertices such that 
\begin{align}\label{eq:soGehtsLos}
 \psi (v_1) = \psi (v_2).
\end{align}
Let $e_1, e_2$ be the edges incident to $v_1$ and $v_2$, respectively. Moreover, let $\Gamma'$ be any graph obtained from $\Gamma$ by changing the lengths of $e_1$ and $e_2$ by less than $\eps$ in such a way that the new lengths satisfy 
\begin{align}\label{eq:crucial}
 L' (e_1) + L' (e_2) = L (e_1) + L (e_2)
\end{align}
and keeping the lengths of all other edges, $L' (e) = L (e)$ for all $e$ with $e \neq e_1, e_2$. Thereby we choose $L' (e_j)$ so close to $L (e_j)$, $j = 1, 2$, that 
\begin{align}\label{eq:close}
 \left| \mu_k (\Gamma) - \mu_k (\Gamma') \right| < \delta
\end{align}
and such that $\mu_k (\Gamma')$ is still simple. Take the eigenfunction $\hat \psi$ on $\Gamma'$ corresponding to $\mu_k (\Gamma')$ and assume for a contradiction that still
\begin{align}\label{eq:stillEqual}
 \hat \psi (v_1) = \hat \psi (v_2).
\end{align}
Consider the graph $\widetilde \Gamma$ obtained from $\Gamma$ by joining the vertices $v_1, v_2$. Note that this metric graph is the same that one gets from joining $v_1$ and $v_2$ in $\Gamma'$ due to~\eqref{eq:crucial}. Moreover, by~\eqref{eq:soGehtsLos} and~\eqref{eq:stillEqual} both functions $\psi$ and $\hat \psi$ can be interpreted as eigenfunctions on $\widetilde \Gamma$ corresponding to the eigenvalues $\mu_k (\Gamma)$ and $\mu_k (\Gamma')$, respectively, and according to variational principles,
\begin{align*}
 \mu_k (\Gamma) = \mu_m (\widetilde \Gamma) \quad \text{and} \quad \mu_k (\Gamma') = \mu_n (\widetilde \Gamma)
\end{align*}
for certain indices $m, n \leq k$. However, comparing this with~\eqref{eq:close} and the initial choice of $\delta$ yields $\mu_k (\Gamma) = \mu_k (\Gamma')$ and thus, $\psi$ and $\hat \psi$ are eigenfunctions on $\widetilde \Gamma$ corresponding to the same eigenvalue. Since $\psi$ and $\hat \psi$ take critical points very close to each other but not at exactly the same point, they are linearly independent, which contradicts the simplicity of the eigenvalues of $\widetilde \Gamma$.

The above argument reduces the number of pairs of boundary vertices on which $\psi$ has equal values by one. We may now apply it inductively to every pair of vertices with equal value, for a successively smaller value of $\varepsilon$ each time which preserves non-equality of all pairs of distinct values, to obtain the conclusion of the lemma.
\end{proof}

We can now proceed to the proof of Theorem~\ref{thm:m-generically-2}.

\begin{proof}[Proof of Theorem~\ref{thm:m-generically-2}]
If $\Gamma$ is a cycle graph then by attaching an arbitrarily small edge to an arbitrary point on the cycle we obtain a lasso graph, which satisfies the assertion of the theorem, see~Example~\ref{ex:lasso}. Otherwise we may assume from the beginning that $\mu_2 (\Gamma)$ is simple and that the corresponding eigenfunction $\psi$ is nonzero on each vertex, see~\cite{BeLi17}. Let $\eps > 0$. First we will attach pendant edges of length less than $\eps / 2$ to each point (without loss of generality a vertex of degree at least two) on $\Gamma \setminus \partial \Gamma$ on which $\psi$ takes its global maximum or minimum and, at the same time, shorten the lengths of all edges incident to that vertex by less than $\eps / 2$, by an amount to be specified precisely later. Let $v \in M (\Gamma) \setminus \partial \Gamma$ and let $e_1, \dots, e_d$ be the edges incident to $v$, assumed to be parametrised away from $v$. As $\psi$ takes its maximum at $v$, one has
\begin{align*}
 \psi_{e_j} (x) = \psi (v) \cos (k x), \qquad x \in [0, L (e_j)], \quad j = 1, \dots, d,
\end{align*}
where $k = \sqrt{\mu_2 (\Gamma)}$. Let $x_0 < L (e_j)$ for $j = 1, \dots, d$, and let $\Gamma'$ be the metric graph obtained from $\Gamma$ by shortening the length of each of the edges $e_1, \dots, e_d$ by $x_0$ and attaching one additional pendant edge $e_{d + 1}$ of some length $\eta$ (to be determined later) to $v$. On the edges $e_j$, parametrised as $[x_0, L (e_j)]$, $j = 1, \dots, d$, we keep the eigenfunction $\psi$ as before, while on $e_{d + 1}$, parametrised away from $v$ as $[0, \eta]$, we define
\begin{align*}
 \psi_{e_{d + 1}} (x) = \psi (v) \cos (k x) - \sum_{j = 1}^d \psi_{e_j}' (x_0) \sin (k x), \qquad x \in [0, \eta],
\end{align*}
and we choose $\eta$ such that $\psi_{e_{d + 1}}' (\eta) = 0$. Then $\eta$ depends smoothly on $x_0>0$ small, with $\eta \to 0$ as $x_0 \to 0$. In particular, we may choose $x_0$ such that both $x_0$ and $\eta$ are less than $\eps$. The function on $\Gamma'$ thus obtained is an eigenfunction of $- \Delta_{\Gamma'}$ corresponding to the eigenvalue $k^2 = \mu_2 (\Gamma)$, and if $\eps$ is chosen sufficiently small, then by simplicity of $\mu_2$ under perturbations and continuity of the low eigenvalues as $\varepsilon \to 0$, see \cite{BLS19}, we have that $\mu_2 (\Gamma') = k^2$ and this eigenvalue is still simple. After applying the same procedure to each point in $M (\Gamma)$ which is not a boundary vertex, we arrive at a graph $\Gamma'$ such that $M (\Gamma') \subset \partial \Gamma'$.

It remains to apply Lemma~\ref{lem:boundary} to $\Gamma'$ and $\eps / 2$ instead of $\eps$ to obtain a graph $\Gamma_\eps$ such that $\mu_2 (\Gamma_\eps)$ is simple and the corresponding eigenfunction has pairwise distinct values at all boundary edges. If the perturbation of the edge lengths is chosen sufficiently small, by continuity we have $M (\Gamma_\eps) \subset \partial \Gamma_\eps$ and, in particular, $|M (\Gamma_\eps)| = 2$.
\end{proof}

\subsection{Graphs with a finite number of hot spots: proof of Proposition~\ref{prop:number}}
\label{sec:number-examples}

Proposition~\ref{prop:number} follows from the following explicit example.

\begin{example}
For any $n \geq 2$, there exists a graph $\Gamma_n$ such that $\mu_2 (\Gamma_n)$ is simple and $|M| = |M_\loc| = n$. Indeed, start with a path graph (an interval) of length $1$ and attach at one end $n-1$ equal edges of length $\varepsilon>0$ each to form $\Gamma_n$; that is, $\Gamma_n$ is an $n$-star with one long and $n-1$ short edges. Denote by $v_0$ the unique vertex of degree $n$. Then by a standard argument, cf.~\cite[Section~5.1]{BKKM18}, we may choose a basis of eigenfunctions on $\Gamma_n$ such that each is either invariant with respect to permutations of the short edges (``even'') or zero at $v_0$ and supported on exactly two short edges (``odd''); moreover, each even eigenfunction is the unique even eigenfunction in its eigenspace (up to scalar multiples). The non-constant even eigenfunction with smallest corresponding eigenvalue is monotonic along the graph, reaching its global minimum at (say) the end of the longer edge and its global maximum at the end of each of the $n-1$ shorter edges. Call its eigenvalue $\mu$. Then one may easily adapt the proof of \cite[Lemma~5.5]{BKKM18} to show that every odd eigenfunction has an eigenvalue strictly larger than $\mu$; hence, $\mu=\mu_2(\Gamma)$ and this eigenvalue is simple; in particular, $|M|=n$, and as $M = \partial \Gamma$, Corollary~\ref{cor:treelocalhotspots} yields $M_\loc = M$.
\end{example}

\section{On the location of the hot spots}
\label{sec:location}

\subsection{Hot spots and bridges: proof of Theorem~\ref{thm:doublyConnectedPart}}
\label{sec:location-bridges}

The next theorem has a number of consequences for the location of the hot spots and will, in particular, lead to the proof of Theorem~\ref{thm:doublyConnectedPart}. 

In what follows, by {\em disconnecting} a vertex $v_0$ of $\Gamma$ we understand the result of replacing $v_0$ by $\deg v_0$ vertices of degree one. More precisely, we replace $\Gamma$ by a graph with vertex set $(\cV \setminus \{v_0\}) \cup \{v_1, \dots, v_{\deg v_0}\}$ and edge set $(\cE \setminus (\cE_{v_0, \rm i} \cup \cE_{v_0, \rm t})) \cup \{e_1, \dots, e_{\deg v_0} \}$, where $v_1, \dots, v_{\deg v_0}$ are vertices of degree one and $e_j$ connects $v_j$ to a vertex which was previously adjacent to $v_0$ in the original graph $\Gamma$, $j = 1, \dots, \deg v_0$. As any interior point of an edge can be interpreted as a vertex of degree two, we speak accordingly of disconnecting arbitrary points on $\Gamma$. 

\begin{theorem}\label{thm:noDisconnect}
Given $\Gamma$, let $\psi$ be any eigenfunction of $- \Delta_\Gamma$ corresponding to the eigenvalue~$\mu_2 (\Gamma)$. Then disconnecting all points in
\begin{align*}
 M^+_{\psi, \loc} := \left\{ x \in \Gamma : \psi (x)~\text{is a nonzero local maximum for}~\psi \right\}
\end{align*}
keeps the graph connected.
\end{theorem}

As the example of the loop shows, removing the set of all nonzero local \emph{extrema} of a given eigenfunction corresponding to $\mu_2 (\Gamma)$ \emph{can} disconnect the graph $\Gamma$. We also point out that the resulting connected graph after disconnecting all nonzero local maxima does not need to be a tree and may also be equal to $\Gamma$ (which is the case exactly when all local maxima of $\psi$ lie on $\partial \Gamma$).

\begin{proof}[Proof of Theorem~\ref{thm:noDisconnect}]
Assume the converse, i.e., after disconnecting all points in $M^+_{\psi, \loc}$ the graph has at least two connected components $\Gamma_1$ and $\Gamma_2$. As $\psi$ is strictly positive on each point of $M^+_{\psi, \loc}$, see Lemma~\ref{lem:maxPos}, the subsets
\begin{align*}
 \Gamma_j^- := \big\{ x \in \Gamma_j : \psi (x) < 0 \big\}, \quad j = 1, 2,
\end{align*}
treated as subsets of $\Gamma$ have positive distance to each other and are nonempty as the restriction $\psi |_{\Gamma_j}$ is a non-constant eigenfunction of $- \Delta_{\Gamma_j}$ and thus has a vanishing integral over $\Gamma_j$, $j = 1, 2$. Define a function $\widetilde \psi : \Gamma \to \C$ by
\begin{align*}
 \widetilde \psi (x) = \begin{cases} c_1 \psi (x), & \text{for}~x \in \Gamma_1^-,\\
                                     c_2 \psi (x), & \text{for}~x \in \Gamma_2^-, \\
                                     0, & \text{otherwise},
                       \end{cases}
\end{align*}
where $c_1, c_2 \in \R \setminus \{0\}$ are chosen such that $\int_\Gamma \widetilde \psi (x) \d x = 0$. Then $\widetilde \psi \in H^1 (\Gamma)$. As $\psi |_{\Gamma_j^-}$ is an eigenfunction of $- \Delta_{\Gamma_j^-}$ corresponding to the eigenvalue $\mu_2 (\Gamma)$ which vanishes on the set 
\begin{align}\label{eq:boundary}
 (\Gamma \setminus \Gamma_j^-) \cap \overline{\Gamma_j^-}
\end{align}
of connection points with the remainder of $\Gamma$, $j = 1, 2$, we have
\begin{align*}
 \int_{\Gamma_j^-} \big|\widetilde \psi' (x) \big|^2 \d x & = c_1^2 \int_{\Gamma_j^-} |\psi' (x)|^2 \d x = c_1^2 \mu_2 (\Gamma) \int_{\Gamma_j^-} |\psi (x)|^2 \d x \\
 & = \mu_2 (\Gamma) \int_{\Gamma_j^-} \big|\widetilde \psi (x) \big|^2 \d x
\end{align*}
for $j = 1, 2$. Hence,
\begin{align*}
 \int_{\Gamma} \big|\widetilde \psi ' \big|^2 \d x & = \int_{\Gamma_1^-} \big|\widetilde \psi' \big|^2 \d x + \int_{\Gamma_2^-} \big|\widetilde \psi' \big|^2 \d x = \mu_2 (\Gamma) \int_\Gamma |\widetilde \psi (x) \big|^2 \d x.
\end{align*}
Together with $\int_\Gamma \widetilde \psi (x) \d x = 0$ this implies that $\widetilde \psi$ is an eigenfunction of $- \Delta_\Gamma$ corresponding to the eigenvalue $\mu_2 (\Gamma)$. In particular, $\widetilde \psi$ satisfies the Kirchhoff condition on both of the sets~\eqref{eq:boundary}. On the other hand, $\widetilde \psi$ has a fixed sign on $\Gamma_1^-$ and on $\Gamma_2^-$, so that each single derivative has to vanish at each point in~\eqref{eq:boundary}, that is, $\widetilde \psi$ has local extrema on all of~\eqref{eq:boundary}. As the values of $\widetilde \psi$ there are all equal to zero, it follows that $\widetilde \psi$ is constantly equal to zero on each edge incident to~\eqref{eq:boundary}. This contradicts the fact that $\psi$ is strictly negative in the interior of $\Gamma_j^-$, $j = 1, 2$.
\end{proof}

Now Theorem~\ref{thm:doublyConnectedPart} follows easily.

\begin{proof}[Proof of Theorem~\ref{thm:doublyConnectedPart}]
Let $x \in M_\loc$. Then there exists an eigenfunction $\psi$ of $- \Delta_\Gamma$ corresponding to $\mu_2 (\Gamma)$ that takes a nonzero local maximum at $x$. In particular, by Theorem~\ref{thm:noDisconnect} removing $x$ from $\Gamma$ does not disconnect the graph, which implies the assertion.
\end{proof}

\subsection{Hot spots and diameter: proof of Propositions~\ref{prop:hotspotsClose} and~\ref{prop:starDiameter}}
\label{sec:location-diameter}

This subsection is devoted to the distance between the ``coldest'' and ``hottest'' spots on a graph. In~\cite[Section~3]{KRpamm} an example of a finite, compact, connected tree graph $\Gamma$ was given for which $\mu_2 (\Gamma)$ is simple and the points realising the maximum and the minimum do not realise the diameter,
\begin{equation}
\label{eq:maxdiamratio}
 \frac{\max \{ \dist (x,y): x,y \in M \}}{\diam(\Gamma)}  < 1.
\end{equation}
In what follows we provide a modification of this example that shows that the ratio in~\eqref{eq:maxdiamratio} can become arbitrarily small. This proves Proposition~\ref{prop:hotspotsClose}.

\begin{example}\label{ex:krpamm}
We are going to construct a tree $\Gamma$ with the properties claimed in Proposition~\ref{prop:hotspotsClose} by means of a splitting procedure. We start with a graph $\Gamma^*$ being an equilateral star with four edges $e_1, e_2, e_3, e_4$ each of which has length $1/2$, around a central vertex $v_0$, see Figure~\ref{fig:example}.
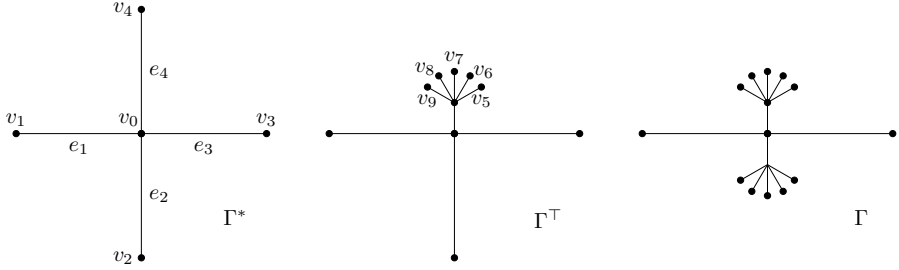
\begin{figure}[htb]
  \centering
  \resizebox{12cm}{!}{
  \begin{tikzpicture}
    %star graph
    \begin{scope}[shift={(6.5,0)}]
    \draw[fill] (-1,0) circle(0.05);
    \draw[fill] (1,0) circle(0.05);
    \draw[fill] (3,0) circle(0.05);
    \draw (-1,0) -- (3,0);
    \draw[fill] (1,-2) circle(0.05);
    \draw[fill] (1,0) circle(0.05);
    \draw[fill] (1,2) circle(0.05);
    \draw (1,-2) -- (1,2);
    \node at (0.8,0) [anchor=south] {$v_0$};
    \node at (-1,0) [anchor=south] {$v_1$};
    \node at (1,-2) [anchor=east] {$v_2$};
    \node at (3,0) [anchor=south] {$v_3$};
    \node at (1,2) [anchor=east] {$v_4$};
    \node at (0,0) [anchor=north] {$e_1$};
    \node at (1,-1) [anchor=west] {$e_2$};
    \node at (2,0) [anchor=north] {$e_3$};
    \node at (1,1) [anchor=west] {$e_4$};
    \node at (2.5,-1.1) [anchor=north] {$\Gamma^*$};
    \end{scope}
    %halftree
    \begin{scope}[shift={(11.5,0)}]
    \draw[fill] (-1,0) circle(0.05);
    \draw[fill] (1,0) circle(0.05);
    \draw[fill] (3,0) circle(0.05);
    \draw (-1,0) -- (3,0);
    \draw[fill] (1,-2) circle(0.05);
    \draw[fill] (1,0) circle(0.05);
    \draw[fill] (1,0.5) circle(0.05);
    \draw (1,-2) -- (1,1);
    \draw[fill] (1.25,0.93) circle(0.05);
    \draw[fill] (1.43,0.75) circle(0.05);
    \draw[fill] (1,1) circle(0.05);
    \draw[fill] (0.75,0.93) circle(0.05);
    \draw[fill] (0.57,0.75) circle(0.05);
    \draw (1,0.5) -- (1.25,0.93);
    \draw (1,0.5) -- (1.43,0.75);
    \draw (1,0.5) -- (0.75,0.93);
    \draw (1,0.5) -- (0.57,0.75);
    \node at (1.43,0.75) [anchor=north] {$v_5$};
    \node at (1.18,0.99) [anchor=west] {$v_6$};
    \node at (1,1) [anchor=south] {$v_7$};
    \node at (0.82,1.04) [anchor=east] {$v_8$};
    \node at (0.57,0.75) [anchor=north] {$v_9$};
    \node at (2.5,-1.1) [anchor=north] {$\Gamma^\top$};
    \end{scope}
    %example tree
    \begin{scope}[shift={(16.5,0)}]
    \draw[fill] (-1,0) circle(0.05);
    \draw[fill] (1,0) circle(0.05);
    \draw[fill] (3,0) circle(0.05);
    \draw (-1,0) -- (3,0);
    \draw[fill] (1,0) circle(0.05);
    \draw[fill] (1,0.5) circle(0.05);
    \draw (1,-1) -- (1,1);
    \draw[fill] (1.25,0.93) circle(0.05);
    \draw[fill] (1.43,0.75) circle(0.05);
    \draw[fill] (1,1) circle(0.05);
    \draw[fill] (0.75,0.93) circle(0.05);
    \draw[fill] (0.57,0.75) circle(0.05);
    \draw (1,0.5) -- (1.25,0.93);
    \draw (1,0.5) -- (1.43,0.75);
    \draw (1,0.5) -- (0.75,0.93);
    \draw (1,0.5) -- (0.57,0.75);
    \draw[fill] (1.25,-0.93) circle(0.05);
    \draw[fill] (1.43,-0.75) circle(0.05);
    \draw[fill] (1,-1) circle(0.05);
    \draw[fill] (0.75,-0.93) circle(0.05);
    \draw[fill] (0.57,-0.75) circle(0.05);
    \draw (1,-0.5) -- (1.25,-0.93);
    \draw (1,-0.5) -- (1.43,-0.75);
    \draw (1,-0.5) -- (0.75,-0.93);
    \draw (1,-0.5) -- (0.57,-0.75);
    \node at (2.5,-1.1) [anchor=north] {$\Gamma$};
    \end{scope}
  \end{tikzpicture}
  }
  \caption{The star graph $\Gamma^*$, the intermediate tree $\Gamma^\top$, and the final, symmetric tree $\Gamma$ in the case $m = 5$.}
  \label{fig:example}
\end{figure}
Then $\mu_2 (\Gamma^*) = \pi^2$ and this eigenvalue has multiplicity 3. In what follows we assume that the edges $e_2, e_4$ are parametrised along the path $\cP$ from $v_2$ to $v_4$, that is, we identify $\cP$ with the interval $[0, 1]$, where $v_0$ is thus identified with $1/2$. Then the function $\psi^*$ acting as $- \cos(\pi x)$ on $\cP$ and being identically equal to zero on the remainder of $\Gamma^*$ is an eigenfunction of $- \Delta_{\Gamma^*}$ corresponding to $\mu_2 (\Gamma^*)$. On the path $\cP \, \widehat{=} \, [0, 1]$ let $x_0 = 1/2 + \eps$ for some $\eps \in (0, 1/2)$. At $x_0$ we split the path (i.e.\ the edge $e_4$) in a balanced way along the eigenfunction to get a new graph $\Gamma^\top$. Namely, we replace $e_4$ by a new edge $\hat e_4$ of length $\eps$ emanating from $v_0$ and $m$ edges $e_5, \dots, e_{4 + m}$ connecting $x_0$ to new boundary vertices $v_5, \dots, v_{4 + m}$, where $m \in \N$ is arbitrary, and we let each of these $m$ edges have length 
\begin{align*}
 L (e_j) = \frac{\arctan\frac{A}{m F}}{\pi}, \quad j = 5, \dots, 4 + m,
\end{align*}
where $A = \sin(\pi x_0)$ and $F = - \cos (\pi x_0)$. We parametrise the edges $e_5, \dots, e_{4 + m}$ outgoing from $x_0$ and define a function $\psi^\top$ on $\Gamma^\top$ by letting $\psi^\top$ be equal to $\psi^*$ on $e_1, e_2, e_3, \hat e_4$ and 
\begin{align*}
 \psi_{e_j}^\top (x) = F \cos (\pi x) + \frac{A}{m} \sin (\pi x), \quad x \in (0, L (e_j)), \quad j = 5, \dots, 4 + m.
\end{align*}
Obviously $\psi^\top$ satisfies $- \psi_e^{\top \prime \prime} = \pi^2 \psi_e^\top$ inside every edge $e$ of $\Gamma$, and it can be checked by calculation that $\psi^\top$ satisfies the  continuity and Kirchhoff conditions at $x_0$ as well as Neumann vertex conditions at $v_5, \dots, v_{4 + m}$. Hence, $\psi^\top \in \ker (- \Delta_{\Gamma^\top} - \pi^2)$. Moreover, $\psi^\top$ takes its maximum (with value $\sqrt{A^2 / m^2 + F^2}$) only at the vertices $v_5, \dots, v_{4 + m}$ and its minimum only at the vertex $v_2$. Moreover, we have
\begin{align*}
 \dist (v_0, v_j) = \eps + \frac{\arctan \frac{A}{m F}}{\pi} = \eps + \frac{\arctan \frac{\tan \big( \pi (1/2 - \eps) \big)}{m}}{\pi},
\end{align*}
and we can make this arbitarily small by choosing $\eps$ sufficiently small and $m$ sufficiently large. In the next step we construct our final tree $\Gamma$ from $\Gamma^\top$ by splitting $e_2$ at $1/2 - \eps$ in $m$ edges along $\psi^\top$ in an analogous way. The resulting function $\psi$ on the tree $\Gamma$ is an eigenfunction of $- \Delta_\Gamma$ corresponding to the eigenvalue $\pi^2$. We claim that $\pi^2$ is still the smallest nontrivial eigenvalue of $\Gamma$; indeed, this follows from \cite[Theorem~3.18(4)]{BKKM18}: since $\Gamma^*$ can be obtained from $\Gamma$ by ``unfolding'' the pendant edges $e_5,\ldots,e_{4+m}$ at $x_0 \widehat{=} 1/2 + \varepsilon$ and the corresponding pendant edges at $1/2-\varepsilon$, we have $\pi^2 = \mu_2 (\Gamma^\ast) \leq \mu_2 (\Gamma)$, and conclude that $\mu_2 (\Gamma) = \pi^2$. Moreover, for any two points $x, y$ on $\Gamma$ such that $\psi$ takes its maximum at $x$ and its minimum at $y$ we have
\begin{align*}
 \dist (x, y) = 2 \eps + 2 \frac{\arctan \frac{\tan \big( \pi (1/2 - \eps) \big)}{m}}{\pi},
\end{align*}
which can be made arbitrarily small by choosing $m$ sufficiently large. On the other hand, $\diam \Gamma = 1$.

Note that $\pi^2$ still has multiplicity $3$ in the spectrum of $\Gamma$, and that $v_1,v_3 \in M(\Gamma)$ still. However, if we lengthen each of the short edges $e_5, \ldots $ by an arbitrary $\delta > 0$ to obtain a new graph $\widetilde{\Gamma}$, then by \cite[Corollary~3.12(2)]{BKKM18}, we have $\mu_2 (\widetilde{\Gamma}) < \mu_2 (\Gamma)$; moreover, one may argue exactly as in \cite[Section~3]{KRpamm} to show that $\mu_2 (\widetilde{\Gamma})$ is simple and its eigenfunction is supported on the complement of $e_1 \cup e_3$: in particular, all points in $M(\widetilde{\Gamma})$ are arbitrarily close to each other.
\end{example}

\begin{proof}[Proof of Proposition~\ref{prop:starDiameter}]
We assume that the star $\Gamma$ consists of at least three edges; otherwise it is a path graph and the claim is clearly true.

Let $\psi$ be as in the proposition. Denote by $v_0$ the star vertex (the only vertex with degree larger than one) and let us first consider the case where $\psi (v_0) = 0$. In this case, if each edge of $\Gamma$ is parametrised from $v_0$ towards the boundary vertex then $\psi_e (x) = A_e \sin (k x)$ holds for every edge $e$, where $A_e \in \R$ and $k = \sqrt{\mu_2 (\Gamma)}$. As $\psi$ has no local extremum in the interior of an edge, see Theorem~\ref{thm:doublyConnectedPart}, but has a vanishing derivative at the boundary vertex, this implies that for each $e \in \cE$ either $L (e) = \pi / (2 k)$ or $A_e = 0$. In other words, each edge in the support of $\psi$ has length $\pi / (2 k)$. On the other hand, each edge on which $\psi$ vanishes identically (if any) has length no greater than  $\pi / (2 k)$: if $e$ is an edge with $L (e) > \pi / (2 k)$ then 
\begin{align*}
 \diam \Gamma \geq \frac{\pi}{2 k} + L (e) > \frac{\pi}{k},
\end{align*}
and this contradicts the fact that $k \leq \pi/\diam \Gamma$ for any tree, see~\cite[Theorem~3.4]{R17}. As $\psi$ takes its maximum and minimum only on $\partial \Gamma$, the assertion of the proposition follows.

Let us now consider the case $\psi (v_0) \neq 0$. Then $\psi_e$ is necessarily non-vanishing for each $e \in \cE$. Observe that in this case $\psi$ has exactly one zero: firstly, it is clear that $\psi$ has at least one zero as $\int_\Gamma \psi \, \d x = 0$. Assume that $\psi$ has two different zeros. These zeros must lie on two different edges as otherwise $\psi$ would have a local extremum inside an edge. But then each of these two edges must have length strictly larger than $\pi / (2 k)$, which implies $\diam \Gamma > \pi / k$, which again contradicts~\cite[Theorem~3.4]{R17}. Thus $\psi$ has exactly one zero. Let us denote the edge that contains the zero by $e_0$. Without loss of generality we may assume that $\psi$ is negative between this zero and the boundary vertex $\hat v$ to which $e_0$ is incident and positive on the remainder of $\Gamma$. In particular, the minimum of $\psi$ is taken on $\hat v$ only, and $\psi$ has nonzero local maxima at all further boundary vertices. Let us assume further that $\Gamma$ is parametrised as a rooted tree with root $\hat v$ in the direction from the root towards the remaining boundary vertices. Let us denote by $e_1, \dots, e_{E - 1}$ the remaining edges. Define
\begin{align*}
 F = \psi_{e_0} (L (e_0)) = \psi (v_0) > 0 \qquad \text{and} \qquad A = \frac{\psi_{e_0}' (L (e_0))}{k} > 0.
\end{align*}
Due to the continuity and Kirchhoff conditions at $v_0$ we have
\begin{align*}
 \psi_{e_j} (x) = F \cos (k x) + A_j \sin (k x), \quad x \in [0, L (e_j)], \quad j =1, \dots, E - 1,
\end{align*}
where $A_1, \dots, A_{E - 1}$ are positive numbers such that $\sum_{j = 1}^{E - 1} A_j = A$. As $\psi_{e_j}$ is monotically increasing on $[0, L (e_j)]$, $j = 1, \dots, E - 1$ and takes its maximum at the endpoint corresponding to $L (e_j)$, it follows by an elementary calculation that the edge length is related to the slope $A_j$ through
\begin{align}\label{eq:lengthAndSlope}
 L (e_j) = \frac{\arctan \frac{A_j}{F}}{k}
\end{align}
and that the maximal value of $\psi_{e_j}$ is
\begin{align}\label{eq:maxValue}
 \psi_{e_j} (L (e_j)) = \sqrt{A_j^2 + F^2}
\end{align}
for $j = 1, \dots, E - 1$. By~\eqref{eq:maxValue}, the latter value is maximised among all these edges if $A_j$ is maximised, and by~\eqref{eq:lengthAndSlope} this is the case if and only if $L (e_j)$ is maximal. Hence, the global maximum of $\psi$ on $\Gamma$ is taken only at boundary vertices corresponding to edges with maximal length among $e_1, \dots, e_{E - 1}$. Any such edge $e_{j_0}$ satisfies $L (e_{j_0}) < \frac{\pi}{2 k} < L (e_0)$ by the above reasoning or, alternatively, by~\eqref{eq:lengthAndSlope}. It follows that
\begin{align*}
 \diam \Gamma = L (e_0) + L (e_{j_0})
\end{align*}
and the right-hand side equals the distance between the unique point $y = \hat v$ where $\psi$ takes its minimum and any point $x$ where $\psi$ takes its maximum. This completes the proof.
\end{proof}

To summarise, in the special case of trees, any global extrema must lie on the boundary, but they do \emph{not} need to lie as far apart from each other as possible. However, the latter is true for any star graph. 

\begin{remark}
\label{rem:flowerDiameter}
The conclusion of Proposition~\ref{prop:starDiameter} also holds for flowers; that is, if $\Gamma$ is any flower, without loss of generality with at least $3$ petals, and $\psi$ is any eigenfunction of $\mu_2 (\Gamma)$ with a global maximum at $x$ and a global minimum at $y$, then $\dist (x,y) = \diam \Gamma$. To see this, simply observe that all eigenfunctions associated with $\mu_2 (\Gamma)$ are (reflection) symmetric with respect to the midpoint of each petal (cf.\ Example~\ref{ex:flower}). Hence, if we fix an eigenfunction $\psi$ associated with $\Gamma$ and denote by $\widetilde \Gamma$ the star formed by replacing each petal $e$ by a pendant edge of length $L (e) / 2$, then by \cite[Corollary~3.6]{BKKM18} we have that $\mu_2 (\widetilde \Gamma) = \mu_2 (\Gamma)$, and $\psi$ may be identified canonically with an eigenfunction $\tilde\psi$ on $\widetilde \Gamma$ corresponding to $\mu_2 (\widetilde \Gamma)$. The result of Proposition~\ref{prop:starDiameter} applied to $\widetilde\Gamma$ and $\tilde\psi$ now yields the corresponding statement for $\Gamma$ and $\psi$.
\end{remark}

\section{Graph topology and hot spots: proof of Theorem~\ref{thm:examplesummary}}
\label{sec:topology}

\begin{proof}[Proof of Theorem~\ref{thm:examplesummary}]
The reasoning is always based on the continuity of the eigenfunctions with respect to varying the edge lengths described in Theorem~\ref{thm:appendix}. The idea is similar in all four cases and we discuss only the first case in full detail. 

(i) If $|\partial \cG| \geq 1$ we choose an edge $e_0$ incident to a vertex of degree one. We obtain an incarnation of $\cG$ as a metric graph $\Gamma (\ell)$ by setting $\ell = \{L (e)\}_{e \in \cE}$, where $L (e_0) = 1$ and $L (e) = \delta$ for a given $\delta > 0$ to be specified shortly and all edges $e \neq e_0$ (cf.\ the appendix). Our aim is to compare the eigenfunction on $\Gamma (\ell)$ corresponding to $\mu_2 (\Gamma (\ell))$ with the respective eigenfunction on the metric graph $\Gamma$ obtained from $\cG$ by setting $\widetilde L (e_0) = 1$ and $\widetilde L (e) = 0$ for all $e \neq e_0$. Then the standard Laplacian on $\Gamma$ can be identified with the Neumann Laplacian on the interval $[0, 1]$ and $\mu_2 (\Gamma) = \pi^2$ is a simple eigenvalue. By Theorem~\ref{thm:appendix} $\mu_2 (\Gamma (\ell))$ is simple for all sufficiently small $\delta$ and for each $\eps > 0$ we may choose $\delta$ so small that the corresponding eigenfunctions $\psi_{2, \ell}$ and $\psi_2$ on $\Gamma (\ell)$ and $\Gamma$, respectively,
satisfy
\begin{align}\label{eq:dasKlappt}
 \sup_{x \in \Gamma (\ell)} \Big| \psi_{2, \ell} (x) - \Big(\cJ_\ell \psi_2 \Big) (x) \Big| < \eps,
\end{align}
where the rescaling operator $\cJ_\ell$ is defined as in~\eqref{eq:Jell}. Note that the function $\psi_2$ can be chosen such that its maximum equals one and is taken at the endpoint corresponding to the vertex of degree one in $\cG$ to which $e_0$ is incident and the function $\cJ_\ell \psi_2$ defined on $\Gamma (\ell)$ is then constantly equal to $- 1$ on every edge apart from $e_0$. Thus, as long as $\eps < \frac{1}{2}$,~\eqref{eq:dasKlappt} implies that the value of $\psi_{2, \ell}$ at the boundary vertex belonging to $e_0$ is larger than $\frac{1}{2}$ while on the edges different from $e_0$ the function $\psi_{2, \ell}$ is bounded by $\frac{1}{2}$. As $\psi_{2, \ell}$ cannot attain its global maximum inside the edge $e_0$ by Theorem~\ref{thm:doublyConnectedPart}, the global maximum must lie on the boundary vertex of $e_0$.

(ii) If $|\partial \cG| \geq 2$ we choose two edges corresponding to vertices of degree one, set the corresponding edge lengths to one and make all other lengths small to obtain  a metric graph $\Gamma (\ell)$ corresponding to $\cG$. We then compare $\Gamma (\ell)$ with the eigenfunction of the graph $\Gamma$ with all lengths zero apart from the two chosen edges that have length one each. This graph can be identified with the interval $[0, 2]$ and the corresponding eigenfunction takes its minimum and maximum on the endpoints of the interval only. Now the eigenfunction on $\Gamma (\ell)$ is a small perturbation of the one on $\Gamma$ and has again minimum and maximum at the two chosen boundary vertices only.

(iii) In the case $\beta \geq 1$ we may assume that also $|\partial \cG| \geq 1$; otherwise we are in the situation of~(iv) as $\cG$ is not a cycle graph. Let $e_0$ be an edge in the doubly connected part and let $e_1$ be an edge incident to a boundary vertex. We choose $\Gamma (\ell)$ to be the metric graph associated with $\cG$ for which $L (e_0) = L (e_1) = 1$ and $L (e) = \delta$ for a sufficiently small $\delta > 0$ for all $e$ different from $e_0$ and $e_1$. Then the graph $\Gamma$ to compare with is the lasso graph consisting of a loop $e_0$ and a boundary edge $e_1$ attached to it, each of them having length one (see also Example~\ref{ex:lasso}). As for this graph the eigenvalue $\mu_2 (\Gamma)$ is simple and the eigenfunction can be chosen to have its maximum at the midpoint of the loop only (and its minimum at the boundary vertex), also on $\Gamma (\ell)$ the second eigenfunction takes its maximum only inside $e_0$.

(iv) Let us now assume $\beta \geq 2$ and let $e_0$, $e_1$ be two edges located in two different cycles of $\cG$. We let $L (e_0) = L (e_1) = 1$ and make all further edges sufficiently small in the graph $\Gamma (\ell)$. Then $\Gamma$ corresponds to the figure-8 graph consisting of two loops of length one each. The corresponding eigenfunction on the figure-8 graph takes, without loss of generality, its minimum at the midpoint of $e_0$ only and its maximum at the midpoint of $e_1$ only. From this the statement for $\Gamma (\ell)$ follows as in the earlier cases.
\end{proof}

\begin{example}
In order to illustrate the proof of Theorem~\ref{thm:examplesummary} we look at the discrete graph $\cG$ consisting of a 3-pumpkin with a boundary edge attached to each of the two vertices, see Figure~\ref{fig:perturbedPumpkin}. 
\begin{figure}[h]
\begin{tikzpicture}
\coordinate (a) at (0,0);
\coordinate (b) at (3,0);
\coordinate (c) at (-2,0);
\coordinate (d) at (5,0);
\draw[fill] (a) circle (1.5pt);
\draw[fill] (b) circle (1.5pt);
\draw[fill] (c) circle (1.5pt);
\draw[fill] (d) circle (1.5pt);
\draw[thick] (a) edge (b);
\draw[thick, bend left=45] (a) edge (b);
\draw[thick, bend right=45] (a) edge (b);
\draw[thick] (c) edge (a);
\draw[thick] (d) edge (b);
\end{tikzpicture}
\caption{The discrete graph $\cG$, a ``pumpkin on a stick''.}
\label{fig:perturbedPumpkin}
\end{figure}
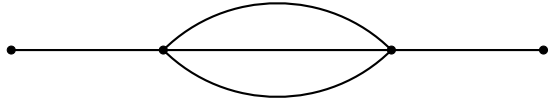
Figure~\ref{fig:abcd} indicates possible choices of edge lengths that lead to different locations of hot spots, either both or one or none of them being on the boundary.
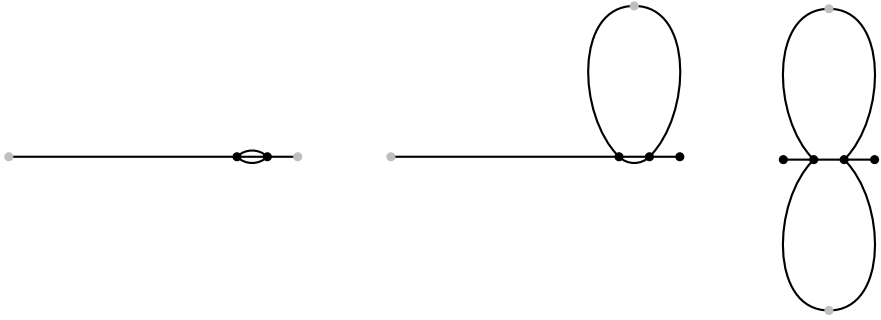
\begin{figure}[h]
\begin{minipage}[b]{3.5cm}
\begin{tikzpicture}
% (i) / (ii)
\coordinate (a) at (0,0);
\coordinate (b) at (0.4,0);
\coordinate (c) at (-3,0);
\coordinate (d) at (0.8,0);
\draw[fill] (a) circle (1.5pt);
\draw[fill] (b) circle (1.5pt);
\draw[thick] (a) edge (b);
\draw[thick, bend left=45] (a) edge (b);
\draw[thick, bend right=45] (a) edge (b);
\draw[thick] (c) edge (a);
\draw[thick] (d) edge (b);
\draw[lightgray, fill] (c) circle (1.5pt);
\draw[lightgray, fill] (d) circle (1.5pt);
\end{tikzpicture}
\end{minipage} \qquad \qquad
\begin{minipage}[b]{3.5cm}
\begin{tikzpicture}
% (iii)
\coordinate (a) at (0,0);
\coordinate (b) at (0.4,0);
\coordinate (c) at (-3,0);
\coordinate (d) at (0.8,0);
\coordinate (h) at (0.2,2);
\draw[fill] (a) circle (1.5pt);
\draw[fill] (b) circle (1.5pt);
\draw[thick] (a) edge (b);
\draw[thick] (a) to[out=135,in=-180] (h);
\draw[thick] (h) to[out=0,in=45] (b);
\draw[thick, bend right=45] (a) edge (b);
\draw[thick] (c) edge (a);
\draw[thick] (d) edge (b);
\draw[lightgray, fill] (c) circle (1.5pt);
\draw[fill] (d) circle (1.5pt);
\draw[lightgray, fill] (h) circle (1.5pt);
\end{tikzpicture}
\end{minipage} \qquad \qquad
\begin{minipage}[c]{1.5cm}
\begin{tikzpicture}
% (iv)
\coordinate (a) at (0,0);
\coordinate (b) at (0.4,0);
\coordinate (c) at (-0.4,0);
\coordinate (d) at (0.8,0);
\coordinate (h) at (0.2,2);
\coordinate (-h) at (0.2,-2);
\draw[fill] (a) circle (1.5pt);
\draw[fill] (b) circle (1.5pt);
\draw[thick] (a) edge (b);
\draw[thick] (a) to[out=135,in=-180] (h);
\draw[thick] (h) to[out=0,in=45] (b);
\draw[thick] (a) to[out=225,in=-180] (-h);
\draw[thick] (-h) to[out=0,in=-45] (b);
\draw[thick] (c) edge (a);
\draw[thick] (d) edge (b);
\draw[fill] (c) circle (1.5pt);
\draw[fill] (d) circle (1.5pt);
\draw[lightgray, fill] (h) circle (1.5pt);
\draw[lightgray, fill] (-h) circle (1.5pt);
\end{tikzpicture}
\end{minipage}
\caption{Metric graph incarnations of the discrete graph~$\cG$ from Figure~\ref{fig:perturbedPumpkin} together with their hot spots (grey).}
\label{fig:abcd}
\end{figure}
\end{example}

\section{Further conjectures and remarks}
\label{sec:conjectures}

We finish by providing a number of further observations, questions and conjectures about properties of the size and location of the sets $M$ and $M_\loc$ that we expect to hold. In many cases, especially with the examples, we strongly expect that with enough effort the ideas presented could be made rigorous, but this would go beyond the scope of this work. Here, as always, we are only interested in graphs satisfying Assumption~\ref{ass}.

\subsection{Conjectures about the number of hot spots}
\label{sec:conj-number}

We start with the cardinality of $M$. We saw in Proposition~\ref{prop:m-finite-or-uncountable} that the set of all possible local extrema, $M_{\loc}$, is always either finite or uncountable (in fact, it always has a finite number of connected components). If we can find eigenfunctions $\psi_1$ and $\psi_2$ whose respective local maxima are close enough together, then we can form appropriate linear combinations having a local maximum in between those of $\psi_1$ and $\psi_2$; this is the statement of Lemma~\ref{lem:Mgeneration}. It seems natural to expect that the same is true globally.

\begin{conjecture}
\label{conj:m-size}
The set $M$ always has a finite number of connected components; in particular, it is either finite or uncountable.
\end{conjecture}

The difficulty lies in controlling the interaction of $\psi_1$ and $\psi_2$ a long way from the maxima in question, as an eigenfunction may potentially have multiple global maxima at considerable distance from each other. Actually, it seems reasonable to ask a stronger question.

\begin{question}
Is it true that $M$ is either finite or equal to $\Gamma$?
\end{question}

The only known examples where $M$ is infinite (equilateral pumpkins, equilateral complete graphs; see Proposition~\ref{prop:m-gamma-examples}) have this property because of a high eigenvalue multiplicity and different eigenfunctions supported throughout the graph. To generate a graph $\Gamma$ for which $M$ is infinite but $M \subsetneq \Gamma$, one would need to find distinct eigenfunctions both (or all) supported within a proper subset of the graph. At least for the standard Laplacian, it seems unlikely that such eigenfunctions could be generated; the particular case of graphs with bridges should be easier to handle.

\begin{conjecture}
Suppose the graph $\Gamma$ has a bridge. Then $M$ and even $M_{\loc}$ are finite.
\end{conjecture}

Note that the presence of a bridge (obviously) does not imply the simplicity of $\mu_2(\Gamma)$, as elementary examples such as stars or star-like graphs show. A related question is a type of inverse problem: if the set $M$ is small, what can we conclude about the eigenfunctions?

\begin{question}
Suppose that $|M|=2$. Does it follow that $\mu_2(\Gamma)$ is simple?
\end{question}

However, it is natural to ask whether introducing delta or other vertex conditions into the mix might change the picture.

\begin{question}
Is it possible to arrange for an infinite set $M$ of all global minima and maxima of second eigenfunctions of $\Gamma$ for which $M \subsetneq \Gamma$, or even a countably infinite $M$, if other vertex conditions than standard are allowed?
\end{question}

Returning to standard vertex conditions, we may in fact ask:

\begin{question}
Are there any graphs $\Gamma$ for which $M=\Gamma$ apart from equilateral pumpkins and complete graphs?
\end{question}

We saw in Theorem~\ref{thm:m-generically-2} above that up to a small modification of the edge lengths and possibly attaching short pendant edges one has $|M| = 2$. Actually, it may be expected that this is a generic property, that is, it can be reached by only perturbing the edge lengths by an arbitrarily small perturbation.

\begin{conjecture}\label{conj:Mgeneric}
Generically, the eigenvalue $\mu_2 (\Gamma)$ is simple and the corresponding eigenfunction has exactly one minimum and one maximum, i.e., $|M| = 2$.
\end{conjecture}

However, producing examples $\Gamma$ where $|M| \geq 3$ should also be possible even in the absence of symmetry or commensurability properties of $\Gamma$.

\begin{conjecture}
There exists a graph $\Gamma$ all of whose edge lengths are pairwise rationally independent but for which $|M|=3$ (or $|M|=n$ for any given $n \geq 2$).
\end{conjecture}

We expect the graph depicted in Figure~\ref{fig:M3} to provide such an example.
\begin{figure}[H]
\begin{tikzpicture}
\coordinate (a) at (-1,0);
\coordinate (b) at (2,0);
\coordinate (c) at (3.5,-1);
\coordinate (d) at (2.5,0.5);
\coordinate (e) at (2.5,1.3);
\coordinate (f) at (3.75,0.5);
\draw[thick] (a) -- (b);
\draw[thick] (b) -- (c);
\draw[thick] (b) -- (d);
\draw[thick] (d) -- (e);
\draw[thick] (d) -- (f);
\draw[lightgray,fill] (a) circle (1.5pt);
\draw[fill] (b) circle (1.5pt);
\draw[lightgray,fill] (c) circle (1.5pt);
\draw[fill] (d) circle (1.5pt);
\draw[fill] (e) circle (1.5pt);
\draw[lightgray,fill] (f) circle (1.5pt);
\node at (a) [anchor=north] {$v_-$};
\node at (b) [anchor=north] {$v_0$};
\node at (c) [anchor=west] {$v_3$};
\node at (e) [anchor=west] {$v_1$};
\node at (f) [anchor=north] {$v_2$};
\node at (0.5,-0.05) [anchor=south] {$e_-$};
\node at (2.35,0.15) [anchor=south east] {$e_0$};
\node at (2.55,0.8) [anchor=east] {$e_1$};
\node at (3.125,0.5) [anchor=south] {$e_2$};
\node at (3,-0.5) [anchor=north east] {$e_3$};
\end{tikzpicture}
\caption{A graph for which we expect that $M=\{v_-,v_2,v_3\}$ is possible even if the edge lengths are incommensurable.}
\label{fig:M3}
\end{figure}
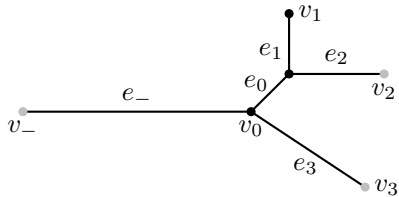
The idea is that $e_-$ should be very long, so that the eigenfunction reaches its unique minimum (say) at $v_-$, and is monotonically increasing on the other edges towards $v_1,v_2,v_3$. For any given edge length $L(e_3)$, it is possible to choose $L(e_0),L(e_1),L(e_2)$ in such a way that the global maximum is at both $v_2$ and $v_3$ (this requires $L(e_2)>L(e_1)$ and $L(e_0)+L(e_2) > L(e_3)$). Since for any value of $L(e_3)$ there exist multiple possible choices of $L(e_0),L(e_1),L(e_2)$ which work, a continuity argument should be able to establish that there is a rationally independent choice which works.

\subsection{Conjectures about the location of the hot spots}
\label{sec:conj-location}

We now formulate a couple of conjectures regarding the location of the hot spots which arise from our considerations above, in particular the question of the distance between them and the diameter of the graph. We recall from Proposition~\ref{prop:hotspotsClose} and Example~\ref{ex:krpamm} that in general there is no relation between diameter and the distance between the hottest and coldest points of the graph. While this example was a tree, it is of course possible to find graphs without any boundary for which the distance between hottest and coldest points is still arbitrarily small compared with the diameter (cf.\ Example~\ref{ex:undUmgekehrt}). But such examples still have an essentially tree-like structure.

\begin{conjecture}
\label{conj:diam-no-boundary}
For every $\varepsilon > 0$ there exists a doubly connected graph $\Gamma$ (cf.\ Section~\ref{sec:location-summary}) with $\diam \Gamma = 1$ such that $\mu_2 (\Gamma)$ is simple and $\max \{\dist (x,y): x,y \in M\}  < \varepsilon$.
\end{conjecture}

A candidate graph, a kind of ``pumpkin necklace'', is depicted in Figure~\ref{fig:double-pumpkin-necklace}.
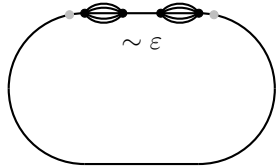
\begin{figure}[H]
\begin{tikzpicture}%[scale=1.2]
\coordinate (a) at (-0.75,0);
\coordinate (b) at (-0.25,0);
\coordinate (c) at (0.25,0);
\coordinate (d) at (0.75,0);
\coordinate (e) at (-1.75,-1);
\coordinate (f) at (1.75,-1);
\coordinate (g) at (-0.75,-2);
\coordinate (h) at (0.75,-2);
\draw[thick] (a) edge (d);
\draw[thick,bend left=60]  (a) edge (b);
\draw[thick,bend left=30] (a) edge (b);
\draw[thick,bend right=30] (a) edge (b);
\draw[thick,bend right=60]  (a) edge (b);
\draw[thick,bend left=60]  (c) edge (d);
\draw[thick,bend left=30] (c) edge (d);
\draw[thick,bend right=30] (c) edge (d);
\draw[thick,bend right=60]  (c) edge (d);
\draw[thick,bend right=45] (a) edge (e);
\draw[thick,bend right=45] (e) edge (g);
\draw[thick] (g) edge (h);
\draw[thick,bend right=45] (h) edge (f);
\draw[thick,bend right=45] (f) edge (d);
\draw[fill] (-0.75,0) circle (1.5pt);
\draw[fill] (-0.25,0) circle (1.5pt);
\draw[fill] (0.25,0) circle (1.5pt);
\draw[fill] (0.75,0) circle (1.5pt);
\node at (0,-0.2) [anchor=north] {$\sim \varepsilon$};
\draw[lightgray,fill] (-0.95,-0.02) circle (1.5pt);
\draw[lightgray,fill] (0.95,-0.02) circle (1.5pt);
\end{tikzpicture}
\label{fig:double-pumpkin-necklace}
\caption{A candidate graph for Conjecture~\ref{conj:diam-no-boundary}. Candidate locations for the hot spots are marked in grey (the precise location will depend on the respective edge lengths and the ``thickness'' of the pumpkins, and could be within the pumpkins).}
\end{figure}
The idea is that the two small but thick equilateral pumpkins are of order $\varepsilon$ apart from each other. If they are thick enough (i.e., have enough parallel edges), then the eigenfunction should take on its maximum on or near one and its minimum on or near the other. On the subject of pumpkins, we briefly mention the following question arising from Example~\ref{ex:pumpkin}: we noted there that $M_{\loc}$ either consists of the midpoints of each edge, or the two points on the longest edge at the correct distance from each other. In fact, for $M$, it seems clear that the global extrema should be at the midpoints of the longest two edges (corresponding to the diameter of the graph). This should be easy but tedious to prove; since we do not need it, we leave it as a conjecture.

\begin{conjecture}
If $\Gamma$ is a (non-equilateral) pumpkin for which $\mu_2 (\Gamma)$ is simple, then $|M|=2$ and $\max \{\dist (x,y): x,y \in M\} = \diam \Gamma$.
\end{conjecture}

Note that in this case we would have $|M_\loc| = E$.

\appendix

\section{Continuity of the eigenfunctions with respect to edge lengths}
\label{sec:continuity}

In this appendix we establish convergence of eigenfunctions in an appropriate sense as the edge lengths of the metric graph $\Gamma$ (as always taken to satisfy Assumption~\ref{ass}) vary and tend to some limit, including possibly length zero in the limit. The convergence of the corresponding eigenvalues is now well known and can be found, for example, in \cite[Appendix~A]{BaLe17} or \cite[Theorem~3.6]{BLS19}; convergence of the eigenvalues and eigenfunctions when the edge lengths remain strictly positive in the limit has been known for longer (see for example \cite{BK12}). While \cite{BLS19} also establishes generalised norm resolvent convergence (again, in an appropriate sense; see also \cite{Cac19} for related results in a somewhat different context), here we will give an elementary proof ``by hand'' that the eigenfunctions, suitably scaled, converge in $H^1$ and hence in $L^\infty$. As a direct consequence hot and cold spots behave continuously with respect to edge length perturbations (although care is needed in the case of eigenvalue crossings).

We will work with the following set-up. Let $\cG$ be a given discrete graph. Consider a corresponding metric graph $\Gamma (\ell)$ obtained by equipping $\cG$ with a vector $\ell = \{L (e) \}_{e \in \cE}$ of positive edge lengths. We are going to send the vector $\ell$ to a limit $\widetilde \ell = \{\widetilde L (e) \}_{e \in \cE}$ that is allowed to contain zero entries and denote the metric graph corresponding to these edge lengths by $\Gamma (\widetilde \ell)$. We stress that the underlying discrete graph of $\Gamma (\widetilde \ell)$ no longer equals $\cG$ if $\widetilde \ell$ has zero entries. In fact, denoting by $\Gamma_+ (\ell)$ the subgraph of $\Gamma (\ell)$ that consists of all edges $e$ such that $\widetilde L (e) > 0$ and by $\cG_+$ the corresponding subgraph of $\cG$, the underlying discrete graph of $\Gamma (\widetilde \ell)$ may be obtained from $\cG_+$ by gluing together pairs of vertices incident with edges to be shrunk to zero. We denote by $\cE_+$ the subset of edges $e$ of $\cG$ for which $\widetilde L (e)$ is positive, i.e.\ the edge set of $\cG_+$. Accordingly we set $\cE_0 = \cE \setminus \cE_+$ and denote by $\Gamma_0 (\ell)$ the subgraph of $\Gamma (\ell)$ formed by the edges in $\cE_0$. To simplify notation we will often suppress the $\widetilde \ell$ and just write $\Gamma$ for the limit graph $\Gamma (\widetilde \ell)$.

Given the $k$-th eigenvalue $\mu_k (\Gamma (\ell))$ of the standard Laplacian on $\Gamma (\ell)$, we will denote by $\psi_{k, \ell}$ a corresponding eigenfunction, normalised so that $\|\psi_{k, \ell}\|_{L^2 (\Gamma (\ell))} = 1$. Likewise, for the eigenvalue $\mu_k (\Gamma)$ on $\Gamma$, $\psi_k$ will denote a corresponding eigenfunction with $\|\psi_k\|_{L^2(\Gamma)} = 1$. Finally, we define a family of scaling operators $\cJ_\ell : H^1 (\Gamma) \to H^1 (\Gamma (\ell))$ as follows. Given a function $f \in H^1 (\Gamma)$, we define
\begin{align}\label{eq:Jell}
 (\cJ_\ell f ) (x) := \begin{cases}
                      \sum_{e \in \cE_+} \chi_e (x) \sqrt{\frac{L (e)}{\widetilde L (e)}} f \Big( \frac{\widetilde L (e)}{L (e)} x \Big), & x \in \Gamma_+ (\ell), \\
                      \sum_{e \in \cE_0} \chi_e (x) f (v_e), & x \in \Gamma_0 (\ell),
                     \end{cases}
\end{align}
where $v_e$ is the vertex of $\Gamma$ into which the edge $e$ collapses; that is, we rescale the edges that tend to a positive limit linearly and extend $f$ by a constant to the edges that are shrunk to zero. We remark that this approach differs slightly from the one used in~\cite[Section~3]{BLS19}. Now we can state the convergence result for the eigenfunctions.

\begin{theorem}
\label{thm:appendix}
With the notation described above, suppose that $\mu_k (\Gamma)$ is a simple eigenvalue of the Laplacian on $\Gamma = \Gamma (\widetilde \ell)$ with standard vertex conditions. Then $\mu_k (\Gamma (\ell))$ is also simple for all $\ell$ sufficiently close to $\widetilde \ell$, $\mu_k (\Gamma (\ell)) \to \mu_k (\Gamma)$ as $\ell \to \widetilde \ell$, and if the normalised eigenfunctions $\psi_{k, \ell}$ are chosen correctly (i.e., with the correct sign), then $\psi_{k, \ell} - \cJ_\ell \psi_k \to 0$ in $H^1 (\Gamma (\ell))$. In particular,
\begin{align*}
 \sup_{x \in \Gamma (\ell)} \big| \psi_{k, \ell} (x) -  \big( \cJ_\ell \psi_k \big) (x) \big| \to 0
\end{align*}
as $\ell \to \widetilde \ell$.
\end{theorem}

\begin{remark}
\label{rem:cont-mult-ev}
If $\mu_k (\Gamma)$ is not simple, say $\mu_k (\Gamma) = \mu_{k+1} (\Gamma)$, then there are two possibilities: either (1) $\mu_k (\Gamma (\ell)) = \mu_{k+1} (\Gamma (\ell))$ for all $\ell$ sufficiently close to $\widetilde \ell$ and the convergence statements of the theorem continue to hold for both $\mu_k$ and $\mu_{k+1}$ and a suitably chosen basis of the space of corresponding eigenfunctions; or (2) $\mu_k (\Gamma (\ell)) < \mu_{k+1} (\Gamma (\ell))$. In the latter case, given $\psi_{k, \ell}$ and $\psi_{k+1, \ell}$, we can \emph{find} a basis of eigenfunctions $\{\psi_k, \psi_{k+1}\}$ for $\mu_k (\Gamma) = \mu_{k+1} (\Gamma)$, such that the conclusions hold for these eigenfunctions. The proofs are essentially the same and we do not go into details.
\end{remark}

Before giving the proof of Theorem~\ref{thm:appendix} we first state a technical result which allows us to control both the $L^\infty$-norm of an eigenfunction and its derivative, as well as their $L^2$-norms localised on a part of the graph, in terms of the eigenvalue.

\begin{lemma}
\label{lem:ef-linfty-control}
Given $\Gamma$, let $\psi$ be an eigenfunction corresponding to the eigenvalue $\mu>0$ of the Laplacian on $\Gamma$ with standard vertex conditions, normalised so that $\|\psi\|_{L^2 (\Gamma)}=1$. Then
\begin{equation}
\label{eq:ef-linfty-control}
	\|\psi\|_{L^\infty (\Gamma)} \leq \sqrt{\mu L (\Gamma)}.
\end{equation}
In particular, if $\Gamma_0 \subset \Gamma$ is an arbitrary subgraph of $\Gamma$, then
\begin{equation}
\label{eq:ef-local-control}
	\|\psi|_{\Gamma_0}\|_{L^2 (\Gamma_0)} \leq \sqrt{L (\Gamma_0) \mu L (\Gamma)}.
\end{equation}
For the derivative $\psi'$ we have
\begin{equation}
\label{eq:deriv-linfty-control}
	\|\psi'\|_{L^\infty(\Gamma)} \leq \mu L (\Gamma) \|\psi\|_{L^\infty (\Gamma)},
\end{equation}
and, with $\Gamma_0$ as above,
\begin{equation}
\label{eq:deriv-local-control}
	\|\psi'|_{\Gamma_0}\|_{L^2 (\Gamma_0)} \leq \sqrt{L (\Gamma_0)} \mu L (\Gamma) \|\psi\|_{L^\infty (\Gamma)} \leq \sqrt{L (\Gamma_0)}(\mu L (\Gamma))^{3/2}.
\end{equation}
\end{lemma}

\begin{proof}
Since $\mu>0$ is not the first eigenvalue, $\psi$ changes sign in $\Gamma$ and so there exist $x_0, x_m \in \Gamma$ such that $\psi(x_0)=0$ and $|\psi(x_m)|=\|\psi\|_{L^\infty(\Gamma)}$; we will assume without loss of generality that $\psi(x_m)>0$, and that $x_0$ and $x_m$ are vertices (possibly of degree two). Now since $\Gamma$ is connected, there exists a non-self-intersecting path $\cP \subset \Gamma$ running from $x_0$ to $x_m$. Hence, by the fundamental theorem of calculus applied to $\psi \in C^\infty (0, L (e))$ along each edge $e$ in $\cP$,
\begin{displaymath}
	\|\psi\|_{L^\infty(\Gamma)} = \psi(x_m) - \psi(x_0) = \int_{\cP} \psi'(x)\,\textrm{d}x \leq \int_{\cP}|\psi'(x)|\,\textrm{d}x
	\leq \sqrt{L (\cP)}\|\psi'\|_{L^2(\cP)},
\end{displaymath}
where the last estimate follows from the Cauchy--Schwarz inequality. Since $\cP$ is non-self-intersecting, $L (\cP) \leq L (\Gamma)$ and 
\begin{displaymath}
	\|\psi'\|_{L^2(\cP)} \leq \|\psi'\|_{L^2(\Gamma)} = \sqrt{\mu}.
\end{displaymath}
This proves the first assertion. For the second, simply use the inequality
\begin{displaymath}
	\|\psi|_{\Gamma_0}\|_{L^2(\Gamma_0)}^2 \leq L (\Gamma_0) \|\psi|_{\Gamma_0}\|_{L^\infty (\Gamma_0)}^2.
\end{displaymath}
For the third, fix any edge $e \simeq [0,L(e)]$; then for any $x,y \in [0,L(e)]$ we have
\begin{displaymath}
	|\psi'(x) - \psi'(y)| = \left| \int_x^y \psi''(t)\,\textrm{d}t\right| = \mu \left|\int_x^y \psi(t)\,\textrm{d}t\right| \leq \mu L (e) \|\psi\|_{L^\infty(0, L (e))}.
\end{displaymath}
This gives us a corresponding bound on the oscillation on that edge,
\begin{equation}
\label{eq:osc-bound}
 \osc_{e} \psi' : = \sup_{x,y \in [0, L (e)]} |\psi'(x)-\psi'(y)| \leq \mu L (e) \|\psi\|_{L^\infty(0, L (e))}.
\end{equation}
Since $\psi'$ is not continuous on $\Gamma$, an additional argument is necessary to obtain \eqref{eq:deriv-linfty-control} from \eqref{eq:osc-bound}. First observe that since $\mu>0$ and hence $\psi$ is not constant, it attains a maximum; moreover, since $\psi'$ is continuous on each (closed) edge, $\|\psi'\|_{L^\infty(\Gamma)}$ is attained. We select a point $v_0$ at which $|\psi'|$ attains its maximum; we assume without loss of generality that $v_0$ is a vertex, possibly of degree two (and the maximum is attained at the endpoint of a given edge $e_0$ incident to $v_0$); we also assume that the edge orientation is such that $v_0 = o (e_0)$ and $\psi|_{e_0}'(0)>0$. Finally, we likewise assume that every local maximum of $\psi$ is a vertex (of degree $\geq 1$). Note that at any local maximum of $\psi$ the Kirchhoff condition guarantees that all derivatives are zero.

By assumption, with our orientation $\psi$ is monotonically increasing on $e_0$. We now consider the closed subset $\widehat\Gamma$ of $\Gamma$ consisting of all (oriented) maximal increasing paths from $v_0$, as follows: we call a non-self-intersecting oriented path starting from $v_0$ (more precisely, starting along $e_0$) \emph{maximal} if $\psi$ is increasing along each edge of the path, and the path ends at a local maximum of $\psi$. Since $\psi$ has at least one maximum, the set of all maximal paths is non-empty; as announced, we denote by $\widehat\Gamma$ their union.

We claim that
\begin{equation}
\label{eq:oscillation-claim}
	\|\psi'\|_{L^\infty(\Gamma)} \leq \sum_{e \subset \widehat{\Gamma}} \osc_{e} \psi'.
\end{equation}
The desired estimate \eqref{eq:deriv-linfty-control} then follows from \eqref{eq:oscillation-claim} and \eqref{eq:osc-bound} by summing the latter over all edges in $\widehat\Gamma$. To prove the claim, we note that the orientation of all paths in $\widehat\Gamma$ is consistent, in the sense that if any two paths share an edge, then that edge carries the same orientation on both paths; moreover, with said orientation, $\psi$ is monotonically increasing on each edge. For each such edge $e \simeq [0,L(e)]$, we will assume the orientation is such that $0$ is the initial vertex and $L(e)$ the terminal vertex, that is, $\psi(0) < \psi(x) < \psi(L(e))$ for all $x \in (0,L(e))$.

As such, for any vertex $v \in \widehat\Gamma$ (excluding $v_0$) all edges in $\widehat\Gamma$ incident to it may be classified as either \emph{incoming} (if the paths are oriented towards $v$ on these edges, equivalently, $\psi$ on these edges is below $\psi(v)$) or \emph{outgoing} ($\psi$ is above $\psi(v)$). Moreover, by construction, if $v \in \widehat\Gamma$ is not a local maximum, then every edge incident to it such that $\psi$ is increasing on that edge away from $v$, that is the derivative on $e$ in the direction towards $v$ is negative, is necessarily in $\widehat\Gamma$, as this set is the union of all possible maximal paths. Hence, for every such $v$, with the notation and orientation as above, the Kirchhoff condition implies that

\begin{equation}
\label{eq:path-condition}
	\sum_{\substack{e \in \cE (\widehat \Gamma) \text{ incoming}\\ \text{at } v}} \psi'(L(e)) \leq \sum_{\substack{e \in \cE (\widehat \Gamma) \text{ outgoing}\\ \text{at } v}} \psi'(0).
\end{equation}
Now by our choice of orientation and the fact that $\psi$ is monotonically increasing (and hence $\psi'$ is positive but monotonically decreasing), on every edge $e$ in $\widehat\Gamma$ we have
\begin{equation}
\label{eq:edge-path-oscillation}
	\psi'(0) = \psi'(L(e)) + \osc_e \psi'.
\end{equation}
An induction argument combining \eqref{eq:path-condition} and \eqref{eq:edge-path-oscillation}, plus the initial condition \linebreak $\psi' |_{e_0} (0)=\|\psi'\|_{L^\infty (\Gamma)}$ at the initial edge $e_0$, and the terminal conditions $\psi'=0$ at the end of every path, proves the claim and hence \eqref{eq:deriv-linfty-control}. Finally, \eqref{eq:deriv-local-control} follows from \eqref{eq:deriv-linfty-control} just as \eqref{eq:ef-local-control} follows from~\eqref{eq:ef-linfty-control}.
\end{proof}

\begin{proof}[Proof of Theorem~\ref{thm:appendix}]
The convergence of the eigenvalues follows directly from the references given at the beginning of the appendix. We have to prove the statements about the eigenfunctions. For this we introduce the following notation: for any $f \in L^2 (\Gamma (\ell))$ we define a function $\widetilde f \in L^2 (\Gamma)$ by
\begin{displaymath}
	\widetilde f (x) := \sum_{e \in \cE_+} \chi_e (x) \sqrt{\frac{\widetilde L (e)}{L (e)}} f \bigg( \frac{L (e)}{\widetilde L (e)} x\bigg), \qquad x \in \Gamma,
\end{displaymath}
that is, $\widetilde f$ is obtained from $f$ by rescaling the function on edges that maintain a positive length in the limit as $\ell \to \widetilde \ell$ and the edges whose lengths converge to zero are cut off from the support of $\widetilde f$. We remark that $f \in H^1 (\Gamma (\ell))$ implies that $\widetilde f$ belongs to $\widetilde H^1 (\Gamma)$ but not necessarily to $H^1 (\Gamma)$. We remark also that for any $f \in H^1 (\Gamma)$ and any $\ell$ we have $\widetilde{\cJ_\ell f} = f$ with $\cJ_\ell$ defined in~\eqref{eq:Jell}. In steps 1--5 of this proof we will show that for the eigenfunctions $\psi_{k, \ell}, \psi_k$ mentioned in the theorem we have 
\begin{align}\label{eq:soGutwieGut}
 \widetilde \psi_{k, \ell} \to \psi_k
\end{align}
in $\widetilde H^1 (\Gamma)$ and, hence, also in $C (0, \widetilde L (e))$ for each edge $e$ of $\Gamma$. In Step~6 we will then obtain the actual statement of the theorem.

\emph{Step 1:} up to a subsequence, $(\widetilde\psi_{k, \ell})$ has a weak limit $\psi^\ast$ in $\widetilde H^1(\Gamma)$ as $\ell \to \widetilde \ell$; convergence is strong in $L^2 (\Gamma)$, and $\|\psi^\ast\|_{L^2 (\Gamma)}=1$. To prove this, first note that (with our normalisation) $\|\psi_{k, \ell}'\|_{L^2(\Gamma (\ell))}^2 = \mu_k (\Gamma (\ell)) \to \mu_k (\Gamma)$, which means, when combined with the definition of $\widetilde\psi_k$ and the convergence of the edge lengths, that the family $(\widetilde\psi_{k, \ell})$ is bounded in $\widetilde H^1 (\Gamma)$ for $\ell$ in a neighbourhood of $\widetilde \ell$. This gives the existence of a weakly convergent subsequence as $\ell \to \widetilde \ell$ (say, to $\psi^\ast$), which is strongly convergent in $L^2 (\Gamma)$ by compactness of the embedding $\widetilde H^1 (\Gamma) \hookrightarrow L^2 (\Gamma)$. Now it follows from a short calculation that
\begin{displaymath}
	\big\|\widetilde \psi_{k, \ell} \big\|_{L^2(\Gamma)}^2 = \int_{\Gamma_+ (\ell)} |\psi_{k, \ell}|^2 \,\textrm{d}x = \|\psi_{k, \ell}\|_{L^2 (\Gamma (\ell))}^2 - \int_{\Gamma_0 (\ell)} |\psi_{k, \ell}|^2\,\textrm{d} x,
\end{displaymath}
and by~\eqref{eq:ef-local-control} we have
\begin{displaymath}
	\int_{\Gamma_0 (\ell)} |\psi_{k, \ell}|^2\,\textrm{d}x \leq L (\Gamma_0 (\ell)) \mu_k (\Gamma (\ell)) L (\Gamma (\ell)) \to 0
\end{displaymath}
as $\ell \to \widetilde \ell$. As the $\psi_{k, \ell}$ are normalised it follows that $\|\psi^\ast\|_{L^2(\Gamma)} = \lim_{\ell \to \widetilde \ell} \|\widetilde \psi_{k, \ell}\|_{L^2 (\Gamma)} = 1$.

\emph{Step 2:} $\psi^\ast \in H^1 (\Gamma)$. Let $e \in \cE_0$, that is, $L(e) \to 0$ as $\ell \to \widetilde \ell$. To show that $\psi^\ast \in H^1 (\Gamma)$, it suffices to show that
\begin{equation}
\label{eq:step2}
	\psi_{k, \ell} (L(e)) - \psi_{k, \ell} (0) \to 0
\end{equation}
as $\ell \to \widetilde \ell$, since then $\psi^\ast$ will satisfy the continuity condition at the corresponding vertex in $\Gamma$ (since then \eqref{eq:step2} will hold for \emph{all} such edges, and $\widetilde\psi_{k, \ell}$ is by construction continuous across every pair of adjacent edges which are not separated by vanishing edges in $\Gamma (\ell)$). To prove \eqref{eq:step2} we argue as in the proof of Lemma~\ref{lem:ef-linfty-control}:
\begin{align}\label{eq:japp}
\begin{split}
	|\psi_{k, \ell} (L(e)) - \psi_{k, \ell} (0)| & \leq \int_0^{L(e)} |\psi_{k, \ell}'|\,\textrm{d} x\leq \sqrt{L(e)}\|\psi_{k, \ell}'\|_{L^2(0, L (e))}\\
	& \leq \sqrt{L (e)} \|\psi_{k, \ell}'\|_{L^2 (\Gamma (\ell))} = \sqrt{L(e)}\sqrt{\mu_k(\Gamma (\ell))} \to 0
\end{split}
\end{align}
since $L(e) \to 0$ and $\mu_k(\Gamma (\ell))$ remains bounded as $\ell \to \widetilde \ell$.

\emph{Step~3:} up to a subsequence,
\begin{align}\label{eq:jetztSchlaegtsDreizehn}
 \big\| \widetilde \psi_{k, \ell} \big\|_{\widetilde H^1 (\Gamma)}^2 \to 1 + \mu_k (\Gamma)
\end{align}
as $\ell \to \widetilde \ell$. Indeed, we already showed that $\|\widetilde \psi_{k, \ell}\|_{L^2 (\Gamma)}^2 \to 1$ in Step~1. Moreover, 
\begin{align*}
 \big\| \widetilde \psi_{k, \ell}' \|_{L^2 (\Gamma)}^2 & = \sum_{e \in \cE_+} \bigg( \frac{L (e)}{\widetilde L (e)} \bigg)^2 \int_0^{L (e)} |\psi_{k, \ell}'|^2 \, \textup{d} x \leq \max_{e \in \cE_+} \bigg( \frac{L (e)}{\widetilde L (e)} \bigg)^2 \int_{\Gamma_+ (\ell)} |\psi_{k, \ell}'|^2 \,\textup{d} x \\
 & = \max_{e \in \cE_+} \bigg( \frac{L (e)}{\widetilde L (e)} \bigg)^2 \bigg( \|\psi_{k, \ell}'\|_{L^2 (\Gamma (\ell))}^2 - \int_{\Gamma_0 (\ell)} |\psi_{k, \ell}'|^2\,\textrm{d} x \bigg)
\end{align*}
and in the same way
\begin{align*}
 \big\| \widetilde \psi_{k, \ell}' \|_{L^2 (\Gamma)}^2 & \geq \min_{e \in \cE_+} \bigg( \frac{L (e)}{\widetilde L (e)} \bigg)^2 \bigg( \|\psi_{k, \ell}'\|_{L^2 (\Gamma (\ell))}^2 - \int_{\Gamma_0 (\ell)} |\psi_{k, \ell}'|^2\,\textrm{d} x \bigg).
\end{align*}
Since
\begin{align*}
 \int_{\Gamma_0 (\ell)} |\psi_{k, \ell}'|^2\,\textrm{d} x \leq L (\Gamma_0 (\ell)) \big( \mu_k (\Gamma (\ell)) L (\Gamma (\ell)) \big)^3 \to 0
\end{align*}
as $\ell \to \widetilde \ell$ by~\eqref{eq:deriv-local-control}, the latter estimates and $\|\psi_{k, \ell}'\|_{L^2 (\Gamma (\ell))}^2 = \mu_k (\Gamma (\ell))$ yield
\begin{align*}
 \big\| \widetilde \psi_{k, \ell}' \|_{L^2 (\Gamma)}^2 \to \mu_k (\Gamma)
\end{align*}
as $\ell \to \widetilde \ell$. This proves~\eqref{eq:jetztSchlaegtsDreizehn}.

\emph{Step~4:} $\psi^\ast = \psi_k$ is the eigenfunction for $\mu_k (\Gamma)$. We prove this by induction on~$k$. When $k=1$ the claim is a direct consequence of the fact that the eigenfunctions are constant: we have
\begin{displaymath}
 \widetilde\psi_{1, \ell}^2 \equiv \frac{1}{L (\Gamma (\ell))}, \qquad \psi_1^2 \equiv \frac{1}{L (\Gamma)},
\end{displaymath}
and $L (\Gamma (\ell)) \to L (\Gamma)$ as $\ell \to \widetilde \ell$. Now suppose the assertion is true for $i = 1, \ldots, k - 1$. Firstly, since up to a subsequence $\widetilde \psi_{k, \ell}$ converges weakly to $\psi^\ast$ in $\widetilde H^1 (\Gamma)$, for this subsequence we have 
\begin{displaymath}
 \|\psi^\ast\|_{\widetilde H^1 (\Gamma)}^2 \leq \liminf_{\ell \to \widetilde \ell} \big\|\widetilde\psi_{k, \ell} \big\|_{\widetilde H^1 (\Gamma)}^2 = 1 + \mu_k (\Gamma)
\end{displaymath}
by Step~3. Since $\psi^\ast \in H^1 (\Gamma)$ with $\|\psi^\ast\|_{L^2(\Gamma)}=1$, by the variational characterisation of $\mu_k (\Gamma)$ it suffices to show that $\psi^\ast$ is orthogonal in $L^2 (\Gamma)$ to the eigenfunctions $\psi_1, \ldots, \psi_{k-1}$. For this, in turn it suffices to show that
\begin{equation}
\label{eq:rescaled-orthogonality}
	\int_\Gamma \widetilde\psi_{i, \ell} \widetilde\psi_{k, \ell} \,\textrm{d}x \to 0
\end{equation}
for all $i=1,\ldots, k - 1$, since then
\begin{displaymath}
	\int_\Gamma \psi_i \psi^\ast\,\textrm{d}x = \lim_{\ell \to \widetilde \ell} \int_\Gamma \psi_i \widetilde\psi_{k, \ell}\,\textrm{d}x
	= \lim_{\ell \to \widetilde \ell} \int_\Gamma \widetilde\psi_{i, \ell} \widetilde\psi_{k, \ell} \,\textrm{d} x = 0,
\end{displaymath}
where we have used the weak convergence of $\widetilde\psi_{k, \ell}$ to $\psi_k$, and the strong convergence of $\widetilde\psi_{i, \ell}$ to $\psi_i$, in $L^2(\Gamma)$. To prove \eqref{eq:rescaled-orthogonality}, fix $i$ and use the definition of the rescaled functions:
\begin{equation}
\label{eq:rescaled-on-gamman}
	\int_\Gamma \widetilde\psi_{i, \ell} \widetilde\psi_{k, \ell}\,\textrm{d}x = \int_{\Gamma_+ (\ell)} \psi_{i, \ell}\psi_{k, \ell} \,\textrm{d}x.
\end{equation}
Since on the other hand by the orthogonality of the eigenfunctions on $\Gamma (\ell)$ we have
\begin{displaymath}
	0 = \int_{\Gamma_+ (\ell)} \psi_{i, \ell} \psi_{k,\ell}\,\textrm{d}x + \int_{\Gamma_0 (\ell)} \psi_{i, \ell}\psi_{k, \ell}\,\textrm{d}x,
\end{displaymath}
and
\begin{displaymath}
	\bigg|\int_{\Gamma_0 (\ell)} \psi_{i, \ell}\psi_{k, \ell}\,\textrm{d}x \bigg|
	\leq \|\psi_{i, \ell}\|_{L^2(\Gamma_0 (\ell))} \|\psi_{k, \ell}\|_{L^2 (\Gamma_0 (\ell))} \to 0
\end{displaymath}
by \eqref{eq:ef-local-control}, together with \eqref{eq:rescaled-on-gamman} this establishes \eqref{eq:rescaled-orthogonality} and hence the claim.

\emph{Step 5:} assertion~\eqref{eq:soGutwieGut}. This is now an easy consequence of the previous steps. Indeed, up to a subsequence $(\widetilde \psi_{k, \ell})$ converges weakly to $\psi_k$ in $\widetilde H^1 (\Gamma)$ by Step~4 and the norms converge to $\sqrt{1 + \mu_k (\Gamma)} = \|\psi_k\|_{\widetilde H^1 (\Gamma)}$. From this we obtain strong convergence of a subsequence in $\widetilde H^1 (\Gamma)$. Since the same argument applies to a subsequence of any sequence in the family $(\widetilde \psi_{k, \ell})$, always with the same limit $\psi_k$ (possibly up to the choice of the sign), this yields the convergence~\eqref{eq:soGutwieGut} in $\widetilde H^1 (\Gamma)$. As the embedding of $H^1 (0, \widetilde L (e))$ into $C (0, \widetilde L (e))$ is continuous on each edge $e$ of $\Gamma$, convergence also in $C (0, \widetilde L (e))$ follows on each edge.

\emph{Step 6:} assertion of the theorem. First note that for $x \in \Gamma_+ (\ell)$ we have
\begin{align*}
 \psi_{k, \ell} (x) - \big(\cJ_\ell \psi_k \big) (x) = \sum_{e \in \cE_+} \chi_e (x) \sqrt{\frac{L (e)}{\widetilde L (e)}} \bigg( \widetilde \psi_{k, \ell} \bigg( \frac{\widetilde L (e)}{ L (e)} x \bigg) - \psi_k \bigg( \frac{\widetilde L (e)}{L (e)} x \bigg) \bigg)
\end{align*}
and thus
\begin{align}\label{eq:ersteHaelfte}
 \sup_{x \in \Gamma_+ (\ell)} \big| \psi_{k, \ell} (x) - \big(\cJ_\ell \psi_k \big) (x) \big| \leq \max_{e \in \cE_+} \sqrt{\frac{L (e)}{\widetilde L (e)}} \sup_{y \in \Gamma} \big| \widetilde \psi_{k, \ell} (y) - \psi_k (y) \big| \to 0
\end{align}
as $\ell \to \widetilde \ell$ by the result of Step~5. In order to show the desired convergence on the edges in $\cE_0$ we assume first for simplicity that $|\cE_0| = 1$ and that $e$ is the unique edge in $\cE_0$. Denote by $v_e$ the vertex corresponding to the zero endpoint of $\Gamma (\ell)$ (assuming without loss of generality that the parametrisation direction is chosen independently of $\ell$) and by $\Gamma_+$ the metric graph obtained by equipping $\cG_+$ with the edge lengths $\widetilde L (e)$, $e \in \cE_+$. Then $\widetilde \psi_{k, \ell} |_{\Gamma_+}$ is continuous on $\Gamma_+$ and for $x \in [0, L (e)]$ we have
\begin{align}\label{eq:SoFaengtEsAn}
\begin{split}
 \big| \psi_{k, \ell} (x) - \big(\cJ_\ell \psi_k \big) (x) \big| & = \big| \psi_{k, \ell} (x) - \psi_k (v_e) \big| \\
 & \leq \big| \psi_{k, \ell} (x) - \psi_{k, \ell} |_e (0) \big| + \big| \widetilde \psi_{k, \ell} |_{\Gamma_+} (v_e) - \psi_k (v_e) \big|.
\end{split}
\end{align}
The second summand converges to zero as $\ell \to \widetilde \ell$ by the result of Step~5. For the first one analogously to the computation~\eqref{eq:japp} we get
\begin{align*}
 \big| \psi_{k, \ell} (x) - \psi_{k, \ell} |_e (0) \big| & \leq \int_0^x \big| \psi_{k, \ell}' (t) \big| \, \textup{d} t \leq \sqrt{L (e)} \sqrt{\mu_k (\Gamma (\ell))} \to 0
\end{align*}
as $\ell \to \widetilde \ell$. Putting this into~\eqref{eq:SoFaengtEsAn} and combining it with~\eqref{eq:ersteHaelfte} we obtain the assertion of the theorem in the case that only one edge length shrinks to zero. The general result can be obtained by applying what we have proved successively.
\end{proof}

\begin{ack}
The work of J.B.K. was supported by the Funda\c{c}\~{a}o para a Ci\^{e}ncia e a Tecnologia, Portugal, via the program ``Investigador FCT'', reference IF/01461/2015, and project PTDC/MAT-CAL/4334/2014. J.R.\ gratefully acknowledges financial support by the grant no.\ 2018-04560 of the Swedish Research Council (VR). The first-named author also wishes to acknowledge the kind hospitality afforded to him during a visit to Stockholm University, where part of this work was carried out. The second-named author is grateful to the University of Lisbon for hospitality during a visit. Finally, both authors would like to acknowledge networking support by the COST Action CA18232.
\end{ack}

\end{document}